\documentclass[12pt,reqno,tbtags]{amsart}

\date{12 August 2014}

\title[SIR epidemic on a random graph with given degrees]
{Law of large numbers for the SIR epidemic on a random graph with given degrees}

\author{Svante Janson, Malwina Luczak, Peter Windridge}

\newcommand\urladdrx[1]{{\urladdr{\def~{{\tiny$\sim$}}#1}}}

\address{Department of Mathematics, Uppsala University, PO Box 480,
SE-751~06 Uppsala, Sweden}
\email{svante.janson@math.uu.se}
\urladdrx{http://www2.math.uu.se/~svante/}

\address{School of Mathematical Sciences, Queen Mary University of London, Mile End Road, London, E1 4NS, UK.}
\email{m.luczak@qmul.ac.uk}
\urladdrx{http://www.maths.qmul.ac.uk/~luczak/}

\address{School of Mathematical Sciences, Queen Mary University of London, Mile End Road, London, E1 4NS, UK.}
\email{p.windridge@qmul.ac.uk}
\urladdrx{http://www.maths.qmul.ac.uk/~pwindridge/}


\usepackage{amsmath,amssymb,amsthm,bbm,fullpage} 

\usepackage{ifdraft}
\usepackage{url}
\usepackage[square,numbers]{natbib}
\bibpunct[, ]{[}{]}{;}{n}{,}{,}

\numberwithin{equation}{section}

\renewcommand\le{\leqslant}
\renewcommand\ge{\geqslant}

\allowdisplaybreaks

\newtheorem{theorem}{Theorem}[section]
\newtheorem{lemma}[theorem]{Lemma}

\newtheorem{corollary}[theorem]{Corollary}

\theoremstyle{definition}
\newtheorem{example}[theorem]{Example}

\newtheorem{remark}[theorem]{Remark}

\newtheorem*{acks}{Acknowledgements}

\theoremstyle{remark}

\newenvironment{romenumerate}[1][0pt]{
\addtolength{\leftmargini}{#1}\begin{enumerate}
 \renewcommand{\labelenumi}{\textup{(\roman{enumi})}}%
 \renewcommand{\theenumi}{\textup{(\roman{enumi})}}%
 }{\end{enumerate}}

\newcounter{oldenumi}
{\setcounter{oldenumi}{\value{enumi}}
\begin{romenumerate} \setcounter{enumi}{\value{oldenumi}}}
{\end{romenumerate}}

\newcounter{thmenumerate}

\newcounter{xenumerate}   

\newcommand{\refT}[1]{Theorem~\ref{#1}}

\newcommand{\refL}[1]{Lemma~\ref{#1}}
\newcommand{\refR}[1]{Remark~\ref{#1}}
\newcommand{\refS}[1]{Section~\ref{#1}}
\newcommand{\refSS}[1]{Section~\ref{#1}}
\newcommand{\refSSS}[1]{Section~\ref{#1}}

\newcommand{\refE}[1]{Example~\ref{#1}}

\newcommand{\refApp}[1]{Appendix~\ref{#1}}

\newcommand\marginal[1]{\marginpar{\raggedright\parindent=0pt\tiny #1}}
\newcommand\REM[1]{{\raggedright\texttt{[#1]}\par\marginal{XXX}}}

\begingroup
  \divide\count255 by 60
  \multiply\count255 by -60
  \advance\count255 by \time
  \ifnum \count255 < 10 \xdef\klockan{\the\count1.0\the\count255}
  \else\xdef\klockan{\the\count1.\the\count255}\fi
\endgroup

\newcommand{\sumki}{\sum_{k=1}^\infty}

\newcommand\set[1]{\ensuremath{\{#1\}}}

\newcommand\xpar[1]{(#1)}
\newcommand\bigpar[1]{\bigl(#1\bigr)}
\newcommand\Bigpar[1]{\Bigl(#1\Bigr)}

\newcommand\lrpar[1]{\left(#1\right)}

\newcommand\abs[1]{|#1|}
\newcommand\bigabs[1]{\bigl|#1\bigr|}
\newcommand\Bigabs[1]{\Bigl|#1\Bigr|}
\newcommand\biggabs[1]{\biggl|#1\biggr|}
\newcommand\lrabs[1]{\left|#1\right|}
\def\rompar(#1){\textup(#1\textup)}    

\newcommand\parfrac[2]{\lrpar{\frac{#1}{#2}}}

\def\xexp(#1){e^{#1}}

\newcommand\ntoo{\ensuremath{{n\to\infty}}}

\newcommand\ktoo{\ensuremath{{k\to\infty}}}

\newcommand\ttoo{\ensuremath{{t\to\infty}}}

\newcommand\downto{\searrow}
\newcommand\upto{\nearrow}
\newcommand\punkt{.\spacefactor=1000}    
\newcommand\iid{i.i.d\punkt}
\newcommand\ie{i.e\punkt}
\newcommand\eg{e.g\punkt}

\newcommand\cf{cf\punkt}
\newcommand{\as}{a.s\punkt}

\newcommand\whp{w.h.p\punkt}
\newcommand\whpx{w.h.p}

\newcommand{\tend}{\longrightarrow}
\newcommand\dto{\overset{\mathrm{d}}{\tend}}
\newcommand\pto{\overset{\mathrm{p}}{\tend}}

\newcommand\op{o_{\mathrm p}}
\newcommand\Op{O_{\mathrm p}}

\newcounter{CC}
\newcounter{cc}

\newcommand\E{\operatorname{\mathbb E{}}}
\renewcommand\P{\operatorname{\mathbb P{}}}

\newcommand\Exp{\operatorname{Exp}}

\newcommand\ga{\alpha}
\newcommand\gb{\beta}
\newcommand\gd{\delta}
\newcommand\gD{\Delta}

\newcommand\gG{\Gamma}

\newcommand\eps{\varepsilon}

\newcommand\cA{\mathcal A}

\newcommand\cE{\mathcal E}

\newcommand\cL{{\mathcal L}}

\newcommand\qq{^{1/2}}

\renewcommand{\=}{:=}

\newcommand\intoi{\int_0^1}
\newcommand\intoo{\int_0^\infty}

\newcommand\oi{[0,1]}
\newcommand\dd{\,d}

\newcommand{\pgf}{probability generating function}

\newcommand\rhs{right-hand side}

\usepackage{enumerate}

\numberwithin{equation}{section}

\newcommand{\Prob}{\mathbb{P}}

\newcommand{\N}{\mathbb{N}}
\newcommand{\R}{\mathbb{R}}

\newcommand{\iEndtt}{\tau^*}

\newcommand{\fpT}{\hat \tau_{\infty}} 

\newcommand{\iUp}{T_0} 

\newcommand{\gS}{g} 

\newcommand{\fXS}{h_\mathrm{S}}
\newcommand{\fXI}{h_\mathrm{I}}
\newcommand{\fXR}{h_\mathrm{R}}
\newcommand{\fXX}{h_X}

\newcommand{\fSv}{v_\mathrm{S}}

\newcommand{\pS}{\mathrm p_{\mathrm{S}}}
\newcommand{\pI}{\mathrm p_{\mathrm{I}}}

\newcommand{\aS}{\alpha_{\mathrm{S}}}
\newcommand{\aI}{\alpha_{\mathrm{I}}}
\newcommand{\aR}{\alpha_{\mathrm{R}}}

\newcommand\nS{n_{\mathrm{S}}}
\newcommand\nI{n_{\mathrm{I}}}
\newcommand\nR{n_{\mathrm{R}}}

\newcommand{\nk}{n_{k}}
\newcommand\nSk{n_{\mathrm{S},k}}
\newcommand\nIk{n_{\mathrm{I},k}}
\newcommand\nRk{n_{\mathrm{R},k}}
\newcommand\nIx[1]{n_{\mathrm{I},#1}}

\newcommand{\mS}{\mu_\mathrm{S}}
\newcommand{\mI}{\mu_\mathrm{I}}
\newcommand{\mR}{\mu_\mathrm{R}}
\newcommand{\mX}{\mu}

\newcommand{\vaS}{\tilde\alpha_{\mathrm{S}}}
\newcommand{\vaI}{\tilde\alpha_{\mathrm{I}}}
\newcommand{\vaR}{\tilde\alpha_{\mathrm{R}}}
\newcommand\vnS{\tilde n_{\mathrm{S}}}

\newcommand{\vnk}{\tilde n_{k}}
\newcommand{\vnSk}{\tilde n_{\mathrm{S},k}}
\newcommand{\vnIk}{\tilde n_{\mathrm{I},k}}
\newcommand{\vnRk}{\tilde n_{\mathrm{R},k}}
\newcommand{\vmS}{\tilde \mu_\mathrm{S}}
\newcommand{\vmI}{\tilde \mu_\mathrm{I}}
\newcommand{\vmR}{\tilde \mu_\mathrm{R}}
\newcommand{\vmX}{\tilde \mu}

\newcommand{\X}[1]{X_{#1}}
\newcommand{\XS}[1]{X_{\mathrm{S},#1}}
\newcommand{\XI}[1]{X_{\mathrm{I},#1}}
\newcommand{\XR}[1]{X_{\mathrm{R},#1}}

\newcommand{\Sv}[1]{S_{#1}}
\newcommand{\Iv}[1]{I_{#1}}
\newcommand{\Rv}[1]{R_{#1}}

\newcommand{\lIv}[1]{\hat I_{#1}}
\newcommand{\lRv}[1]{\hat R_{#1}}

\newcommand{\lTP}[1]{\theta_{#1}}

\newcommand{\Xtt}[1]{X'_{#1}}
\newcommand{\XStt}[1]{X'_{\mathrm{S},#1}}
\newcommand{\XItt}[1]{X'_{\mathrm{I},#1}}
\newcommand{\XRtt}[1]{X'_{\mathrm{R},#1}}

\newcommand{\Svtt}[1]{S'_{#1}}
\newcommand{\Svktt}[1]{\Svtt{#1}(k)}
\newcommand{\Ivtt}[1]{I'_{#1}}
\newcommand{\Rvtt}[1]{R'_{#1}}

\newcommand{\Rzero}{\mathfrak{R}_0}
\newcommand{\tRzero}{\widetilde{\mathfrak{R}_0}}

\newcommand{\lttInv}[1]{\hat A_{#1}}
\newcommand{\ltt}{\hat \tau}
\newcommand{\xtau}{\bar \tau}

\newcommand{\fXIroot}{\lTP{\infty}}

\newcommand{\calibS}{s_0}

\newcommand\taz{\hat \tau_0}

\newcommand\xnit{\XI{t}}
\newcommand\xnio{\XI{0}}
\newcommand\xnt{\X{t}}
\newcommand\xno{\X{0}}
\newcommand\znt{Z_{t}}
\newcommand\nn{^{(n)}}
\newcommand\nsk{\nSk}
\newcommand\gas{\aS}

\newcommand\Gr{G}
\newcommand\GrCM{G^*}
\newcommand\Grxx{\bar G^*}

\newcommand\xI{x_{\mathrm{I}}}

\newcommand\tth{\tilde\theta}
\newcommand\FF{{}_2F_1}
\newcommand\ypgf[1]{g_{#1}}
\newcommand\hZ{W}
\newcommand\GWp{Galton--Watson process}
\newcommand\rw{\hat Z}
\newcommand\WW{\widehat W}
\newcommand\Wa{W_1}
\newcommand\Wb{W_2}
\newcommand\cAa{\cA_1}
\newcommand\cAb{\cA_2}
\newcommand\cAloop{\cA'}
\newcommand\cApair{\cA''}
\newcommand\QQ{\varkappa}

\begin{document}

\begin{abstract}
We study the susceptible-infective-recovered (SIR) epidemic on a random graph
chosen uniformly subject to having given vertex degrees.  In this model infective vertices
infect each of their susceptible neighbours, and recover, at a
constant rate.

Suppose that initially there are only a few infective vertices.
We prove there is a threshold for a parameter involving the rates and
vertex degrees below which only a small number of infections occur.  Above the
threshold a large outbreak occurs with probability bounded away from zero.  Our main result is that,
conditional on a large outbreak, the evolutions of certain quantities of interest,
such as the fraction of infective vertices, converge to deterministic
functions of time.

We also consider more general initial conditions for the epidemic, and derive criteria for
a simple vaccination strategy to be successful.

In contrast to earlier results for this model, our approach only requires basic regularity conditions and a
uniformly bounded second moment of the degree of a random vertex.

En route, we prove analogous results for the epidemic on the configuration model multigraph
under much weaker conditions.  Essentially, our main result requires only that the initial
values for our processes converge, i.e.\ it is the best possible.
\end{abstract}

\maketitle

\section{Introduction}
\label{s:intro}

The Markovian SIR process is a simple model for a disease spreading around a finite population
in which each individual is either susceptible, infective or recovered.
Individuals are represented by vertices in a graph $\Gr$ with edges corresponding to potentially infectious
contacts.  Infective vertices become recovered at rate $\rho \ge 0$ and infect each neighbour at rate $\beta > 0$;
those are the only possible transitions, \ie{} recovered vertices never become infective.

The applicability, behaviour and tractability of the model depends heavily on how $\Gr$ is chosen.
In classical formulations $\Gr$ is the complete graph (see \cite{daleygani} for a historical account)
but gradually attention has shifted towards more realistic models where
individuals may vary in how many contacts they have.

Particular interest has focused on the case that $\Gr$ itself is random.  Several families
of random graph have been considered, such as Erd\H{o}s--R\'enyi $G(n,p)$ graphs \cite{neal2003sir},
those with local household structure \cite{Ball201053} and other forms of clustering \cite{britton2008epidemics}, see the recent survey \cite{HouseSurvey}.

The present paper concerns SIR epidemics on random graphs with a given degree sequence.  These random graphs are commonly used to model the internet, scientific collaboration networks and sexual contact networks \cite{NewmanStrogatzWatts,Newman_2002,Ball200869} (and the references therein).
Random graphs with given degree sequence are normally constructed via the configuration model,
introduced by Bollob\'as, see~\cite{bollobas}.
In recent years, various properties of these graphs have been studied, such as the appearance and size of a giant component~\cite{molloy1995critical, molloy1998size,JansonLuczakGiantCpt}, as well as the near-critical behaviour~\cite{KangSeierstad, JosephCritical}.  Other quantities investigated include the size of the $k$-core~\cite{JansonLuczakkCore, janson2008asymptotic, Sato}, diameter~\cite{FernholzRamachandran}, chromatic number~\cite{frieze2007chromatic} and matching number~\cite{BohmanFrieze}.

There have been a number of studies of SIR epidemics on random graphs with a given degree sequence.
A set of non-linear ordinary differential equations summarising the time evolution of the epidemic were obtained heuristically by Volz~\cite{Volz_2007}.
Another non-rigorous derivation of these equations is given in~\cite{Miller}.

Decreusefond et al.~\cite{DDMT12} 
study a measure-valued process describing the degrees of susceptible
individuals and the number of edges between different types of
vertices. They prove that, as the population size grows to infinity, the
measure-valued process converges to a deterministic limit, from which the
Volz equations may be derived as a corollary.  The results in~\cite{DDMT12}
are proven under the conditions that the fifth moment of the degree of a
random vertex is uniformly bounded, and that, asymptotically,
the proportion of vertices infective at time zero is positive.

Bohman and Picollelli~\cite{BohmanPicollelli} study the SIR process dynamics on the configuration model with bounded vertex degrees, starting from a single infective. They use a multitype branching process approximation
for both the early and final stages of the epidemic.  The middle phase of the epidemic, while there are at least a moderate number of infectives, is analysed using Wormald's differential
equations method.

Barbour and Reinert \cite{BarbourReinert} use multitype branching process approximations to prove results approximating the entire course of an SIR epidemic within a more general non-Markovian framework, allowing degree dependent infection and recovery time distributions. A result for  graphs with a given degree sequence with bounded vertex degrees follows as a corollary.

See also~\cite{CD-SIS} for the SIS epidemic process on a random graph with given degrees, which exhibits very different behaviour compared to the SIR epidemic studied here.

\textbf{Our contribution.}
In this paper we analyse the SIR epidemic on graphs with a given degree sequence for an arbitrary number of initially infective vertices, assuming only basic regularity conditions and uniform boundedness of the second moment of the degree distribution.  This contrasts with the earlier works mentioned above, which require either uniform boundedness of the fifth moment \cite{DDMT12}, or uniformly bounded degrees \cite{BohmanPicollelli,BarbourReinert}.
In our proof we study the configuration model epidemic under the weaker condition that
the degree of a randomly chosen susceptible vertex is uniformly integrable.  This is the best possible condition for a `law of large numbers' result in the spirit of \cite{Volz_2007,Miller,DDMT12,BohmanPicollelli}, since it amounts to convergence
of the average number of susceptible contacts at the epidemic's epoch; see \refR{r:UI}.
Our approach extends techniques of~\cite{JansonLuczakkCore,JansonLuczakGiantCpt} and leads to fairly simple proofs.

\smallskip

The rest of the paper is laid out as follows.  In \refS{s:notationresults}, we define the model and notation, and state our assumptions and results.
In \refS{s:fastepidemic} we consider a time-changed version of the epidemic, as a tool to be used in our proofs.
In \refS{s:proofs}, we prove our results for multigraphs with a given degree
sequence defined by the configuration model. In \refS{s:branch}, we study
more carefully the probability of a large outbreak and the size of a small
outbreak, obtaining more detailed forms of statements in
Theorem \ref{t:mI0}\ref{t:mI0.i} and~\ref{t:mI0.ii.c}.
In \refS{s:simple}, we transfer the results from multigraphs to simple
graphs with a given degree sequence. In \refS{ss:nosecondmoment} we discuss
briefly what happens when the second moment of the degree of a random vertex
is not uniformly bounded.
\refS{s:time-shift} contains a few remarks on the
random time shift used in our proof.
\refApp{Atimechange} contains a technical lemma on the time change in
\refS{s:fastepidemic}.
\refApp{Anotation} is a summary of the main notation used in the paper.

\medskip

\begin{acks}
The research of M.L. and P.W. was supported by EPSRC grant EP/J004022/2.
The research of S.J. was partly supported by the Knut and Alice Wallenberg Foundation.
S.J. thanks Tom Britton  and
P.W. thanks Thomas House, Pieter Trapman and Viet Chi Tran for useful comments.
This research was initiated following discussions at the ICMS Workshop on
`Networks: stochastic models for populations and epidemics' in Edinburgh 2011.
\end{acks}

\section{Model, notation, assumptions and results}\label{s:notationresults}

For $n \in \N$ and a sequence $(d_i)_1^n$ of non-negative integers,
let $\Gr = \Gr(n, (d_i)_1^n )$ be a simple graph (\ie{} with no loops or double edges) on $n$ vertices, chosen uniformly at random from among all graphs with degree sequence $(d_i)_1^n$.
(We tacitly assume that there is some such graph, so $\sum_{i = 1}^n d_i$ must be even, at least.)

Given the graph $\Gr$, the SIR epidemic evolves as a continuous-time Markov chain.  At any
time, each vertex is either susceptible, infected or recovered.  Each infective vertex recovers at rate $\rho \ge 0$ and also infects each susceptible neighbour at rate $\beta > 0$.

We assume that there are initially $\nS$, $\nI$, and $\nR$ susceptible, infective and recovered vertices, respectively.
Further, we assume that, for each $k \ge 0$, there are respectively $\nSk$, $\nIk$ and $\nRk$ of these vertices with degree $k$.
Thus, $\nS + \nI + \nR = n$ and $\nS = \sum_{k=0}^\infty \nSk$, $\nI = \sum_{k=0}^\infty \nIk$, $\nR = \sum_{k=0}^\infty \nRk$.
We write $n_k$ to denote the total number of vertices with degree $k$; thus, for each $k$,  $n_k =  \nSk + \nIk + \nRk$.
Note that all these parameters, as well as the sequence $(d_i)_1^n$, depend on the number
$n$ of vertices, although we omit explicit mention of this in the notation.  For technical reasons, note that they do not have to be defined for all integers $n$; a subsequence is enough.

We consider asymptotics as $n \to \infty$, and all unspecified limits below are as $n \to \infty$. Throughout the paper we use the notation $\op$ in a standard way.  That is, for a sequence of
random variables $(Y\nn)_1^\infty$ and real numbers
$(a_n)_1^\infty$, `$Y\nn = o_p(a_n)$'  means $Y\nn/a_n \pto 0$.
Similarly, $Y\nn=\Op(1)$ means that for every $\eps>0$ there exists $K_\eps$
such that $\P(|Y\nn|>K_\eps)<\eps$ for all $n$.
For a sequence $(Y\nn_t)_1^\infty$ of real-valued stochastic processes defined on a subset $E$ of $\R$ and a real-valued function $y$ on $E$,
`$Y\nn_t \pto y(t)$ uniformly on $E' \subseteq E$' means $\sup_{t \in E'}\abs{Y\nn_t - y(t)}
\pto 0$.
Given a sequence of events $(\cE_n)_1^\infty$, event $\cE_n$ is said to hold \whp{} (with high probability) if the probability of $\cE_n$ converges to 1.

We assume the following regularity conditions for the degree sequence asymptotics.

\begin{enumerate}
 \renewcommand{\theenumi}{(D\arabic{enumi})}
 \renewcommand{\labelenumi}{\theenumi}

 \item \label{d:alphaprops} \label{d:first}
The fractions
of initially susceptible, infective and recovered vertices converge to some $\aS, \aI,\aR \in [0,1]$, i.e.
\begin{equation}\label{e:alphaprops}
\nS/n \to \aS, \qquad \nI/n \to \aI, \qquad \nR/n \to \aR.
\end{equation}
Further, $\aS > 0$.

\item \label{d:asympsuscdist}  The degree of a randomly chosen susceptible vertex converges to a probability distribution $(p_k)_{0}^\infty$, i.e.
\begin{equation}\label{e:nSktopk}
\nSk/\nS \to p_k, \qquad k \ge 0.
\end{equation}
Further, this limiting distribution has a finite and positive mean
\begin{equation}\label{e:meansuscdist}
\lambda \= \sum_{k = 0}^\infty k p_k \in (0, \infty).
\end{equation}

\item \label{d:suscmeanUS}
The average degree of a randomly chosen susceptible vertex converges to $\lambda$, i.e.
\begin{equation}
 \label{e:suscmeanUS}
 \sum_{k=0}^{\infty} k \nSk/\nS \to \lambda.
\end{equation}

\item
\label{d:meanIRconv} The average degree over all vertices converges to $\mX > 0$, i.e.
\begin{equation}\label{e:meanX0}
 \sum_{k = 0}^\infty k n_k/n =  \sum_{i = 1}^n d_i/n \to \mX,
\end{equation}
and, in more detail, for some $\mS, \mI, \mR$,
\begin{equation}\label{e:meanSs}
\sum_{k = 0}^\infty k \nSk/n \to \mS,
\end{equation}
\begin{align}\label{e:meanXIR}
  \sum_{k = 0}^\infty k \nIk/n \to \mI,
&\qquad  
  \sum_{k = 0}^\infty k \nRk/n \to \mR.
\end{align}

\item \label{d:SI2ndmoment}
The maximum degree of the initially infective vertices is not too large:
\begin{equation}\label{e:SI2ndmomentx}
\max\set{k:\nIk>0}=o(n).
\end{equation}

\item \label{d:p1orrhopos} Either $p_1 > 0$ or $\rho > 0$ or $\mR > 0$.

\label{d:last}
\end{enumerate}

\begin{remark}
Obviously, $\aS+\aI+\aR=1$ and $\mS+\mI+\mR=\mX$.  Further, assumptions \ref{d:alphaprops}--\ref{d:suscmeanUS} imply $\sum_{k = 0}^\infty k \nSk/n \to \aS \lambda$. Thus,
$\mS = \aS \lambda$ and \eqref{e:meanSs} is redundant.  

The assumptions $\aS > 0$ in \ref{d:alphaprops} and $\lambda  > 0$ in \ref{d:asympsuscdist} mean that there are initially a
significant number of susceptibles with non-zero degree.  They are included to avoid trivialities, and, in particular,  imply that $\nS \ge 1$ for large enough $n$.
\end{remark}

\begin{remark}\label{r:UI}
Assumptions \ref{d:alphaprops}--\ref{d:suscmeanUS}
together imply that $\sum_{k = 0}^\infty k \nSk/n$ is
uniformly summable, \ie{} for any $\eps > 0$ there exists $K$ such that $\sum_{k = K+1}^\infty k \nSk / n < \eps$ for $n$ large enough.
Conversely,
\ref{d:alphaprops},
\ref{d:asympsuscdist}  and uniform summability
of $\sum_{k = 0}^\infty k \nSk/n$
imply \ref{d:suscmeanUS}.
\end{remark}

\begin{remark}
In particular, the uniform summability in \refR{r:UI} implies that
$\max\set{k:\nSk>0}=o(n)$.
This and
assumption \ref{d:SI2ndmoment} imply, using \eqref{e:meanX0},
\begin{equation}\label{e:SI2ndmoment}
\sum_{k = 0}^\infty k^2 (\nSk + \nIk) \le
o(n)\sum_{k = 0}^\infty k \nk =
o(n^2).
\end{equation}
Conversely, \eqref{e:SI2ndmoment} implies \eqref{e:SI2ndmomentx}.
We do not need the corresponding condition for initially recovered vertices,
but since these only play a passive role, it would be essentially no loss of
generality to assume that the maximum degree of all vertices
$\max_i d_i = \max\set{k:n_k>0}=o(n)$.
\end{remark}

It will be convenient for us to work with multigraphs, that is to allow
loops and multiple edges. Let $G^* (n, (d_i)_1^n )$ be the
random multigraph with given degree sequence $(d_i)_1^n$ defined by the
configuration model: we take a set of $d_i$ half-edges for each vertex $i$
and combine half-edges into edges by a uniformly random matching (see \eg{}
\cite{bollobas}). Conditioned on the multigraph being simple, we obtain  $G
= G (n, (d_i)_1^n )$, the uniformly distributed random graph with degree
sequence $(d_i)_1^n$.

The configuration model 
has been used in the study of epidemics in a number of earlier works, see, for example,~\cite{andersson1998limit,Ball200869,BrittonJansonMartinLof,DDMT12,BohmanPicollelli}.

We prove our results for the SIR epidemic on $\GrCM$, and, by conditioning on $\GrCM$ being simple, we deduce that these results also hold for the SIR epidemic on $\Gr$ .
Our argument relies on the probability that
$\GrCM$ is simple
being bounded away from zero as  $n \to \infty$.
By the main theorem of \cite{Janson:2009:PRM:1520305.1520316} this occurs provided the following condition holds.
\begin{enumerate}
 \renewcommand{\theenumi}{(G\arabic{enumi})}
 \renewcommand{\labelenumi}{\theenumi}
\item\label{d:sumsquaredi=On}
The degree of a randomly chosen vertex has a bounded second moment, i.e.
\begin{equation}\label{e:sumsquaredi=On}
\sum_{k = 0}^\infty k^2 \nk = O(n). 
\end{equation}
\end{enumerate}

\begin{remark}\label{r:2ndmoment}
Assumption \ref{d:sumsquaredi=On} implies that the distribution $(p_k)_{0}^\infty$ has a finite second moment,
\ie{} $\sum_{k = 0}^\infty k^2 p_k < \infty$.
Note also that \ref{d:sumsquaredi=On} implies
\eqref{e:SI2ndmoment}
and thus \ref{d:SI2ndmoment}.
\end{remark}

\begin{remark}\label{r:simplegraphwithoutG1}
  Although we use \ref{d:sumsquaredi=On} in order to draw conclusions for
  the simple graph $G$, we suspect that the results hold even without it.
Bollob\'as and Riordan \cite{BRsimple} have recently shown results for a
related problem (the size of the giant component in $G$) from the multigraph
case without using \ref{d:sumsquaredi=On}; they show that even if the
probability that the multigraph is simple is almost exponentially small, the
error probabilities in their case are even smaller. We have not attempted
doing anything similar here.
\end{remark}

We study the SIR epidemic on the multigraph $G^*$, revealing its edges dynamically while the epidemic spreads. To be precise, we
call a half-edge free if it is not yet paired to another half-edge.
We start with $d_i$ half-edges attached to vertex $i$ and all half-edges free.
We call a half-edge susceptible,
infective or recovered according to the type of vertex it belongs to.

Now, each free
infective half-edge chooses a free half-edge at rate $\beta$, uniformly at random from among all the other free half-edges.
Together the pair form an edge, and are removed from the pool of free
half-edges.
If the chosen half-edge belongs to a susceptible vertex then that vertex becomes infective.
Infective vertices also recover at rate $\rho$.

We stop the above process when there are no free
infective half-edges,  at which point
the epidemic stops spreading.  Some infective vertices may remain but
they will recover at \iid{}  exponential times without affecting
any other vertex.  In any case, they turn out to be irrelevant for our purposes.
Some susceptible and recovered half-edges may also remain, and these are
paired off uniformly at time $\infty$
to reveal the remaining edges in $\GrCM$.
This step is unimportant for the spread of the epidemic, but we perform it for the purpose of transferring our results to the simple graph $G$.

Clearly, if all the pairings are completed then the resulting graph is the multigraph $G^*$.
Moreover, the quantities of interest (numbers of susceptible, infective and recovered vertices at each time $t$) have the same distribution as if we were to reveal the multigraph $G^*$ first and run the SIR epidemic on $G^*$ afterwards.

For $t \ge 0$, let $\Sv{t}$, $\Iv{t}$ and $\Rv{t}$ denote the numbers of susceptible,
infective and recovered vertices, respectively, at time $t$.
Thus $\Sv{t}$ is decreasing and $\Rv{t}$ is increasing.
Also $\Sv{0} = \nS$, $\Iv{0} = \nI$ and $\Rv{0} = \nR$.

For the dynamics described above (with half-edges paired off dynamically, as needed), for $t \ge 0$, let $\XS{t}$, $\XI{t}$ and $\XR{t}$ be the number of free susceptible, infective and recovered half-edges
at time $t$, respectively.  Thus $\XS{t}$ is decreasing,
$\XS{0} = \sum_{k = 0}^\infty k \nSk$,  $\XI{0} = \sum_{k = 0}^\infty k
\nIk$ and $\XR{0} = \sum_{k = 0}^\infty k\nRk$.
The variables $\XS{t}$, $\XI{t}$ and $\XR{t}$ are convenient tools for the analysis of
$\Sv{t}$, $\Iv{t}$ and $\Rv{t}$, but not `observable' quantities. (They have no interpretation for the version of the SIR process on $G^*$ where the multigraph is constructed upfront.)
For the uniformly random graph $G$ with degree sequence $(d_i)_1^n$, the variables $\XS{t}$, $\XI{t}$ and $\XR{t}$, for $t \ge 0$, are defined as above conditioned on the final multigraph $G^*$ being a simple graph.

\subsection{Results}
\label{s:results}

We will show that, upon suitable scaling, the processes $\Sv{t}, I_t, R_t$, $\XS{t}, \XI{t}, \XR{t}$
converge to deterministic functions.  The limiting functions
will be written in terms of a parameterisation
$\theta_t \in [0,1]$ of time solving an ordinary differential equation
given below.  
The function $\theta_t$ can be interpreted as the limiting probability that a given
initially susceptible half-edge has not been paired with a (necessarily infective) half-edge
by time $t$.
This means that the probability that a given degree $k$ initially susceptible
vertex is still susceptible at time $t$ is asymptotically close to $\theta_t^k$.  With this in mind,
we define the function $\fSv$ by 
\begin{equation}\label{e:fSv}
 \fSv(\theta) \= \aS \sum_{k=0}^\infty p_k \theta^k, \qquad \theta \in [0,1],
\end{equation}
so the limiting fraction of susceptible vertices is $\fSv(\theta_t)$ at time $t$.  Similarly, for the number of susceptible half-edges we define
\begin{equation}\label{e:fXS}
  \fXS(\theta) \=  \aS \sum_{k=0}^\infty k \theta^k p_k = \theta\fSv'(\theta), \qquad \theta \in [0,1].
\end{equation}
For the total number of free half-edges, we let
\begin{equation}
\fXX(\theta) := 
\mX \theta^2, \qquad \theta \in [0,1].
\label{e:fX}
\end{equation}
The intuition here is that two free half-edges disappear
each time an edge is formed by pairing, so a random free half-edge is paired with
intensity twice the intensity of a susceptible free half-edge, and so
the probability that a given half-edge is
still free at time $t$ is asymptotically close to $\theta_t^2$.
For the numbers of half-edges of the remaining types,
we define (with justification in the proof below), for $\theta \in [0,1]$,
\begin{align}
  \fXR(\theta) &:= \mR \theta + \frac{\mX\rho}{\beta}\theta(1-\theta), \label{e:fXR} \\
  \fXI(\theta) &:= \fXX(\theta) - \fXS(\theta) - \fXR(\theta). \label{e:fXI}
\end{align}
Thus $\fXX(\theta) = \fXS(\theta)+\fXI(\theta)+ \fXR(\theta)$.
The corresponding limit functions for infective and recovered vertices are more
easily described by differential equations, which will be introduced
in~\eqref{e:dlIvtTmIg0} and~\eqref{e:dlIvtT}.
Note that
\begin{align}
&&\fSv(1)&=\aS, &
\\ \label{e:fX1}
\fXS(1)&=\aS\lambda=\mS,&
\fXR(1)&=\mR,&
\fXI(1)&=\mX-\mS-\mR=\mI.
\end{align}

We also introduce the `infective pressure'
\begin{equation}\label{e:pI}
\pI(\theta) \= \frac{\fXI(\theta)}{\fXX(\theta)},
\end{equation}
which appears in the differential equations~\eqref{e:dlTPtmIg0}
and~\eqref{e:dlTPt} below.

Our first two theorems concern the case where the initially infective population is macroscopic, so
that the course of the epidemic is approximately deterministic for a long time, until shortly before extinction.

\begin{theorem}\label{t:mIg0}
Let us consider the SIR epidemic on the multigraph $\GrCM$
with degree sequence $(d_i)_1^n$.  Suppose that
\ref{d:first}--\ref{d:last}
are satisfied. Let $\mI > 0$.

\begin{enumerate}
 \renewcommand{\theenumi}{\textup{(\alph{enumi})}}
 \renewcommand{\labelenumi}{\theenumi}

 \item \label{t:mIg0.a}
There is a unique $\fXIroot \in (0,1)$ with $\fXI(\fXIroot) = 0$.  Further,
$\fXI$ is strictly positive on $(\fXIroot,1]$
 and strictly negative on $(0,\fXIroot)$.
 \item \label{t:mIg0.b}
There is a unique continuously differentiable function
  $\lTP{t}:[0,\infty) \to (\fXIroot,1]$ such that
\begin{equation}\label{e:dlTPtmIg0}
\frac{d}{dt} \lTP{t} = -\beta \lTP{t} \pI(\lTP{t}), 
\qquad \lTP{0} = 1.
\end{equation}
Furthermore, $\lTP{t}\downto \fXIroot$ as $t\to\infty$.
\item \label{t:mIg0.c}
Let $\lIv{t}$ be the unique solution to
\begin{equation}\label{e:dlIvtTmIg0}
 \frac{d}{dt} \lIv{t} = \frac{\beta \fXI(\lTP{t})\fXS(\lTP{t}) }{\fXX(\lTP{t})} - \rho \lIv{t}, \; t \ge 0,\qquad \lIv{0} = \aI,
\end{equation}
and $\lRv{t} \= 1 - \fSv(\lTP{t}) - \lIv{t}$.
Then, uniformly on $[0,\infty)$,
\begin{align}
\label{e:convmIg0SIR}
\Sv{t}/n &\pto \fSv(\lTP{t}), &
 \Iv{t}/n &\pto \lIv{t}, & \Rv{t}/n &\pto \lRv{t},
\\
\label{e:convmIg0X}
\XS{t}/n &\pto \fXS(\lTP{t}), & \XI{t}/n &\pto \fXI(\lTP{t}),
&\XR{t}/n &\pto \fXR(\lTP{t}),
\end{align}
and, consequently, $\X{t}/n \pto \fXX(\lTP{t})$.

\item \label{t:mIg0.d}
Hence, the number $\Sv{\infty} \= \lim_{t \to \infty} \Sv{t}$ of susceptibles that escape infection satisfies
\[
\Sv{\infty}/n \pto \fSv(\fXIroot).
\]
\end{enumerate}
\end{theorem}

\begin{theorem}\label{t:mIg0-1}
Let us consider the SIR epidemic on the
uniform simple graph $\Gr$
with degree sequence $(d_i)_1^n$.  Suppose that
\ref{d:first}--\ref{d:last} and \ref{d:sumsquaredi=On}
are satisfied. Let $\mI > 0$.
Then the conclusions of Theorem~\ref{t:mIg0} hold.
\end{theorem}

Decreusefond et al.~\cite{DDMT12} obtain a related result, assuming that the fifth moment of the degree of a random vertex is uniformly bounded as $n \to \infty$.

\begin{remark}
We can give examples in which $\Sv{t}/n$, $\Iv{t}/n$ and $\Rv{t}/n$ fail to
converge to deterministic limits when assumption \ref{d:SI2ndmoment} does
not hold.  Generally, we believe \ref{d:SI2ndmoment} is necessary for the
convergence to deterministic limits  in Theorems \ref{t:mIg0} and \ref{t:mIg0-1}, but have not
attempted to prove it.
\end{remark}

Our third and fourth theorems concern the case where there are initially a small number of infectives.
Let
\begin{equation}\label{e:R0}
\Rzero \= \parfrac{\beta}{\rho + \beta} \parfrac{\aS}{\mX}
  \sum_{k=0}^\infty (k-1) k p_k;
\end{equation}
this quantity can be interpreted as the basic reproductive ratio of the epidemic.
When $\Rzero > 1$, then
there is a positive probability that a large epidemic develops in the population,
as previously identified in the literature on epidemic models, such as~\cite{AnderssonSocialNetworks,Newman_2002,Volz_2007,BohmanPicollelli}.
In that event, once established, the evolution of the epidemic is
approximately deterministic, as in~Theorems \ref{t:mIg0} and~\ref{t:mIg0-1}.

We will prove below that there is a unique $\fXIroot \in (0,1)$ with
$\fXI(\fXIroot) = 0$, and that if there is a large epidemic, the number of
susceptibles that never get infected is approximately $n\fSv(\fXIroot)$.
Fix a number $s_0 \in (\fSv(\fXIroot), \aS)$, i.e., between the (approximate)
fractions of susceptibles at the beginning and at the end of
the epidemic in the case that a large epidemic develops,
and let
\begin{equation}\label{e:T0}
\iUp \= \inf\{ t \ge 0: \Sv{t} \le n\calibS \}.
\end{equation}
(This means that $\iUp=\infty$ if $S_t$ never falls below $n\calibS$.
We will see that this corresponds to the case of a small outbreak.)
We shift the initial condition of the limiting differential equation,
now defined on $(-\infty, \infty)$,
so that $t = 0$ corresponds to the time $\iUp$ in the random process, by which
the fraction of susceptible individuals has fallen from about $\aS= \fSv (1)$ to some
fixed smaller $\calibS$. By time $\iUp$, a positive fraction of the population
has been infected, and from that point onwards the quantities of interest
follow a law of large numbers. The exact choice of $s_0$ is unimportant.

We extend the processes to be defined on
$(-\infty,\infty)$ by taking $\Sv{t} = \Sv{0}$ for $t < 0$, and similarly
for the other processes.

\begin{theorem}\label{t:mI0}
Let us consider the SIR epidemic on the multigraph $\GrCM$
with degree sequence $(d_i)_1^n$.
Suppose that \ref{d:first}--\ref{d:last} and \ref{d:sumsquaredi=On} are satisfied.
Suppose also that $\aI=\mI = 0$ but there is initially at least one infective vertex with non-zero degree.  

\begin{enumerate}
 \renewcommand{\theenumi}{\textup{(\roman{enumi})}}
 \renewcommand{\labelenumi}{\theenumi}

\item \label{t:mI0.i}
If $\Rzero \le 1$ then the number $\nS - \Sv{\infty}$ of initially susceptible vertices that ever get infected is $\op(n)$.

\item \label{t:mI0.ii}
Suppose $\Rzero > 1$.
\begin{enumerate}
 \renewcommand{\theenumii}{\textup{(\alph{enumii})}}
 \renewcommand{\labelenumii}{\theenumii}

 \item \label{t:mI0.ii.a}
There is a unique $\fXIroot \in (0,1)$ with $\fXI(\fXIroot) = 0$.  Further, $\fXI$ is strictly positive on $(\fXIroot,1)$
 and strictly negative on $(0,\fXIroot)$.

 \item \label{t:mI0.ii.b}
Let $\calibS \in (\fSv(\fXIroot),\fSv(1))$.
Then there is a unique continuously differentiable $\lTP{t}:\R \to (\fXIroot,1)$
such that
\begin{equation}\label{e:dlTPt}
\frac{d}{dt} \lTP{t} = -\beta \lTP{t} \pI(\lTP{t}),
\qquad \lTP{0} = \fSv^{-1}(\calibS).
\end{equation}
Furthermore, $\lTP{t}\downto \fXIroot$ as $t\to\infty$ and $\lTP{t}\upto 1$ as $t\to -\infty$.

\item \label{t:mI0.ii.c}
Let $\iUp$ be defined by \eqref{e:T0}.
Then
$\liminf_{n \to  \infty} \Prob(\iUp < \infty) > 0$.
Furthermore, if
 the initial number of infective half-edges
$\xnio \to \infty$, then $\Prob(\iUp <\infty) \to 1$.

\item \label{t:mI0.ii.d}
Let $\lIv{t}$ be the unique solution to
\begin{equation}\label{e:dlIvtT}
 \frac{d}{dt} \lIv{t} = \frac{\beta \fXI(\lTP{t})\fXS(\lTP{t}) }{\fXX(\lTP{t})} - \rho \lIv{t}, \qquad \lim_{t \to -\infty} \lIv{t} = 0,\\
\end{equation}
and $\lRv{t} \= 1 - \fSv(\lTP{t}) - \lIv{t}$.

Conditional on $\iUp < \infty$, then,
uniformly on $(-\infty,\infty)$,
\begin{align}
\Sv{\iUp+t}/n &\pto \fSv(\lTP{t}), &
\Iv{\iUp+t}/n &\pto \lIv{t}, & \Rv{\iUp+t}/n &\pto \lRv{t},
\label{e:convscSIR} \\
\XS{\iUp+t}/n &\pto \fXS(\lTP{ t}),& \XI{\iUp+t}/n &\pto \fXI(\lTP{t}), &
\XR{\iUp+t}/n &\pto \fXR(\lTP{t}),
\label{e:convscX}
\end{align}
and, consequently, also 
$\X{\iUp+t}/n \pto \fXX(\lTP{t})$.

\item \label{t:mI0.ii.e}
Conditional on $\iUp < \infty$, the number of susceptibles
that escape infection
satisfies
\[
\Sv{\infty}/n \pto \fSv(\fXIroot).
\]

\item \label{t:mI0.ii.f}
The number of susceptibles that ever get infected $\Sv0-\Sv\infty$
satisfies $\Sv0-\Sv\infty=\op (n)$ on the event $\iUp = \infty$, in the sense that, for all $\eps > 0$, $\Prob (\iUp = \infty, \Sv0-\Sv\infty > \eps n ) = o(1)$ as $n \to \infty$.

Similarly,
$X_{S,0}-X_{S,\infty}=\op (n)$, $\sup_{t \ge 0} \XI{t}=\op (n)$,
$\sup_{t \ge 0} (\X0-\X{t})=\op (n)$ on $\iUp = \infty$.
\end{enumerate}
\end{enumerate}

The same result holds 
even without assumption
\ref{d:sumsquaredi=On}, except that, in this case,  it is possible to have
$\lTP{t}:\R \to (\fXIroot,1]$ with $\lTP{t}=1$ for $t\le \hat{A}_0$, for some $\hat{A}_0 < 0$.
\end{theorem}

\begin{theorem}\label{t:mI0-1}
Let us consider the SIR epidemic on the
uniform simple graph $\Gr$
with degree sequence $(d_i)_1^n$.  Suppose that
\ref{d:first}--\ref{d:last} and \ref{d:sumsquaredi=On}
are satisfied. Let $\mI= 0$.
Then the conclusions of Theorem~\ref{t:mI0-1} hold.
\end{theorem}

It  is possible that
$\P(\iUp = \infty)\to0$, and then the statements in
Theorem \ref{t:mI0}\ref{t:mI0.ii.f}
are trivial. In order to prove that quantities of interest  are small
conditional on $\iUp=\infty$ in this case, we would need to study the speeed
at which $\P(\iUp = \infty)\to0$. Nevertheless, the
theorem shows a dichotomy when $\Rzero>1$: \whp{} either $\iUp=\infty$ and the
outbreak is small, with only a few individuals infected;
or $\iUp<\infty$ and the outbreak is large,
with $(\alpha_S-\fSv(\fXIroot))n+o(n)$ individuals infected
(a more detailed description of the evolution is given in \ref{t:mI0.ii.d}).

In fact, we will take $\iUp<\infty$
as the definition of a large outbreak.
(Formally this depends on the choice of $s_0$, but the theorem shows that
any two choices \whp{} yield the same result.)
Thus the probability $\P(\iUp<\infty)$ in \ref{t:mI0.ii.c} is, by definition,
the probability of a large outbreak.  We give a formula for this probability in \refT{T:GW1},
using a branching process approximation
to the early stage of the epidemic
(or an equivalent approximation using a random walk),
see further Sections \ref{s:large-epidemic} and \ref{s:branch}.
The condition $\Rzero > 1$ 
can be interpreted as supercriticality of this branching process approximation.

Furthermore, it turns out that a small outbreak is really small. In \refT{T:GW2}, we sharpen Theorem \ref{t:mI0.ii.f}
by showing that, for a small outbreak,
only $\Op(1)$ susceptibles are infected,
both in the (sub)critical case ($\Rzero \le 1$) provided $\xnio=O(1)$,
and in the supercritical case ($\Rzero > 1$).

For bounded degree sequences and $\Rzero \neq 1$, a result similar to~Theorems \ref{t:mI0} and~\ref{t:mI0-1} is proven in~\cite{BohmanPicollelli} by Bohman and Picollelli.
The threshold in $\Rzero$ for a possible large outbreak and the final size of a large outbreak are derived heuristically in \cite{AnderssonSocialNetworks, Newman_2002, Volz_2007}.
All the above papers assume that initially there are no recovered vertices.
Our motivation for allowing the presence of initially recovered individuals
is to be able to analyse simple vaccination strategies, see \refS{s:vaccination}.  Before
doing that, we give some connections to related results.

\medskip

Theorems \ref{t:mIg0} and \ref{t:mIg0-1} imply ${\XI{t}}/{\X{t}} \pto \pI(\theta_t)$ and
${\XS{t}}/{\X{t}} \pto \pS(\theta_t)$ uniformly when $\mI > 0$, where
$\pS(\theta)\=\fXS(\theta)/\fXX(\theta)$
is defined analogously to $\pI$ in \eqref{e:pI}. Theorems  \ref{t:mI0} and \ref{t:mI0-1} yield
the same result for the time shifted process when $\mI = 0$ and $\Rzero
>1$, conditional on a large outbreak. To explain the connection with \cite{Volz_2007}, let
\begin{equation}
\gS(\theta)\= \sum_{k=0}^\infty p_k \theta^k, \qquad \theta \in [0,1],
\end{equation}
the probability generating function for the asymptotic degree
distribution of initially susceptible vertices.
Note that $\fSv(\theta) = \aS \gS(\theta)$
and $\fXS(\theta) = \aS \theta \gS'(\theta)$.
Differentiating $\pI(\theta_t)$ and $\pS(\theta_t)$ yields,
using \eqref{e:dlTPtmIg0} and \eqref{e:fXS}--\eqref{e:fXI},
\begin{align}
\frac{d\pI(\theta_t)}{dt} &= \pI(\theta_t)\left( -(\rho + \beta) + \beta \pI(\theta_t) + \beta  \pS(\theta_t)\theta_t\frac{\gS''(\theta_t)}{\gS'(\theta_t)}\right),
\\
\frac{d\pS(\theta_t)}{dt}& = \beta \pI(\theta_t) \pS(\theta_t)\left(1 -
\theta_t \frac{\gS''(\theta_t)}{\gS'(\theta_t)}\right).
\end{align}
These are the `Volz equations' \cite[Table 3]{Volz_2007} mentioned in
the introduction.
Volz \cite{Volz_2007} derived them heuristically,
assuming that the number of edges from a newly infective vertex to susceptible, infective
and recovered vertices has multinomial distribution with parameters $\pI$, $\pS$ and $1-\pS - \pI$.

\refT{t:mI0-1} relates to the existence of a giant
component in $\Gr$ as follows.  The epidemic
spreads only within connected components of $\Gr$.  Further, if there are no recoveries,
then all vertices connected to an initially infective vertex eventually get infected.
Indeed, when $\rho = \mI = \mR = 0$, the threshold $\Rzero > 1$ is equivalent to
$\sum_{k = 0}^\infty k (k - 2)p_k > 0$; this
is the well known condition of Molloy and Reed \cite{molloy1995critical}
for existence of a giant component.  Also,
in part \ref{t:mI0.ii.a}, the equation defining $\fXIroot \in (0,1)$ becomes
$\lambda \fXIroot^2 - \sum_{k = 0}^\infty k p_k\fXIroot^k  = 0$ in this case.
With this value of $\fXIroot$,
and assuming $\aI=\aR=0$ so $\aS=1$,
it is known that
$1 - \fSv(\fXIroot) = 1 - \sum_{k = 0}^\infty p_k\fXIroot^k$ is the
fraction of vertices in the giant
component
\cite{molloy1998size} (see also \cite{JansonLuczakGiantCpt}).

The connection to the giant component explains why \ref{d:p1orrhopos}
is needed, at least when $\Rzero = 1$. 
Suppose that~\ref{d:p1orrhopos} is not satisfied, \ie{} both $\rho = \mR = 0$ and $p_1 = 0$.
If also $\mI = 0$ 
then $\Rzero = 1$ is equivalent to $\sum_{k = 0}^\infty k(k-2)p_k = 0$, and
so only $p_0$ and $p_2$ can be non-zero.
At least three different types of behaviour of component sizes in $\Gr$ are possible in this case.  We will demonstrate
them with the following examples from
\cite[Remark 2.7]{JansonLuczakGiantCpt}, see also \refR{r:h=0}.  We assume that $\nR = 0$ in each example.

The first example is a random 2-regular graph, that is $n_2 = n$ for all $n$. In this case, all the components are cycles.
Let $V_1 \ge V_2 \ge $ denote the ordered component sizes.
Then $V_1/n$ converges weakly to a non-degenerate distribution on $[0,1]$,
and the same holds for $V_2/n$, $V_3/n$, and so on \cite[Lemma 5.7]{ArratiaBarbourTavare}.
Let us suppose that there is initially a lone infective vertex.
The number of vertices in the component it occupies (and hence eventually infects) is a size biased sample from $(V_1, V_2, \ldots)$, and, divided by $n$, also converges to a non-degenerate
distribution on $[0,1]$.


For the second example, we suppose that $n = n_1+n_2$, where $n_2/n \to 1$, $n_1$ is even and $n_1 \to \infty$.
The desired graph can be obtained from a random 2-regular graph with $n-n_1/2$ vertices by
selecting $n_1/2$ degree 2 vertices at random, one after another, and creating
two degree 1 vertices out of each one.  During this procedure, components are chosen in a size-biased fashion
and split uniformly
(except on the first attempt, since they were all cycles to begin with).
The largest component in the resulting graph contains
only $o_p(n)$ vertices, and so only $o_p(n)$ susceptibles are infected if $\nI$ is bounded.

For our third example, we take $n = n_2 + n_4$, where
$n_2/n\to 1$ and $n_4 \to \infty$. Each vertex of degree 4 can be obtained by merging a pair of vertices of degree 2. Analogously to the previous example,
we see there is a unique giant component with $n - o_p(n)$ vertices.
Hence, even if only a single given
vertex is initially infective, then all $n - o_p(n)$ susceptibles
in the giant component succumb to infection \whpx.


\subsection{Vaccination}\label{s:vaccination}

Let us suppose that we vaccinate some susceptible vertices before the epidemic process starts.
The vaccine is assumed perfect, so that a vaccinated vertex never becomes infective.
In particular, vaccinated vertices behave like recovered vertices in the SIR dynamics.
Let us use this fact to analyse degree dependent vaccination strategies by applying
Theorems \ref{t:mI0} and \ref{t:mI0-1}
to a suitably modified degree sequence.

We assume that each initially susceptible vertex
of degree $k \ge 0$ is vaccinated with probability $\pi_k \in [0,1)$, independently of all the others.  
Here are two examples of such strategies.

\smallskip

\emph{Uniform vaccination.} We vaccinate every susceptible vertex with the same probability $\pi_k = v$ for all $k$ and some  $v\in [0,1)$, independently of all the others.  The total number $V$ of vaccinations thus satisfies $V/n_{\mathrm S} \pto v$, using the law of large numbers.

\smallskip
\emph{Edgewise vaccination.} We vaccinate the end point
of each susceptible half edge with probability $v \in [0,1)$, independently of all the other half-edges.
Thus the probability that a degree $k$ susceptible is vaccinated is $\pi_k:= 1 - (1 - v)^k$, and, under our assumptions, the total number
$V$ of vaccinations satisfies $V/n_{\mathrm{S}} \pto \sum_{k=0}^\infty p_k
\pi_k$.

\medskip

These strategies are considered in \cite{BrittonJansonMartinLof}, where their efficacy in a related epidemic
model (equivalent to constant recovery times) is compared, along with two other strategies (uniform
acquaintance vaccination strategy and another edgewise strategy, neither of which can be studied with the present argument).

As noted above, vaccinating a vertex amounts to changing its type from susceptible to recovered.
Let us calculate the (random) number of vertices of each type post-vaccination, and show that
assumptions \ref{d:first}--\ref{d:last} hold for the modified degree sequence.

We add a tilde to our notation for the post-vaccinated epidemic.  Thus $\vnSk$ denotes
the number of degree $k \ge 0$  initially susceptible vertices that remain unvaccinated.
We have $\vnSk \sim  $ Binomial$(n_{\mathrm S,k}, 1 - \pi_k)$, and, by the law of
large numbers,
\[
\vnSk
= \nSk(1 - \pi_k) + o_p(n)
= n \aS p_k(1 - \pi_k) + o_p(n).
\]
Using the uniform summability of $\sum_{k=0}^\infty k \nSk/n$ (see \refR{r:UI})
\[
\vnS := \sum_{k=0}^\infty \vnSk = n \aS \sum_{k=0}^\infty p_k (1 - \pi_k) + o_p(n),
\]
whence
\[
\frac{\vnS}{n} \pto \aS \sum_{k=0}^\infty p_k (1 - \pi_k) =: \vaS > 0.
\]
Furthermore,
\[
\frac{\tilde n_{\mathrm{S},k}}{ \tilde n_{\mathrm{S}} } \pto \frac{p_k(1-\pi_k)}{\sum_{k=0}^\infty p_k (1 - \pi_k)} =: \tilde p_k,
\]
and $\tilde p_1 > 0$ if $p_1 > 0$.   The mean of $(\tilde p_k)_0^\infty$ is
\[
\tilde \lambda := \sum_{k=0}^\infty k \tilde p_k \le \aS\lambda/\vaS < \infty.
\]
By the uniform summability of $\sum_{k=0}^\infty k \nSk/n$,
and the inequality $\vnSk \le \nSk$, we also have
\[
\frac{1}{ \vnS }\sum_{k=0}^\infty k \vnSk \pto \sum_{k=0}^\infty k \tilde p_k.
\]

The initial number $\vnRk = \nRk + (\nSk - \vnSk)$ of degree $k$ vertices that
are either recovered or have been vaccinated satisfies
\[
\vnRk = \nRk + n\aS \pi_k p_k + o_p(n),
\]
and so, again using the uniform summability of $\sum_{k=0}^\infty k \nSk/n$,
\[
\sum_{k=0}^\infty k\vnRk/n\to
\vmR := \mR + \aS \sum_{k=0}^\infty k \pi_k p_k.
\]

The number $\vnIk$ of infective vertices of degree $k$ after the
vaccinations is unchanged, \ie{} $\vnIk = \nIk$. It follows that
\[
 \vaI = \lim_{ n \to \infty } \sum_{k=0}^\infty \vnIk/n = \aI,
\qquad  \vmI = \lim_{ n \to \infty } \sum_{k=0}^\infty k \vnIk/n = \mI.
\]
Similarly, $\vnk=\nk$, and so $\vmX=\mX$. Furthermore,
\begin{equation*}
 \sum_{k=0}^\infty k \vnSk/n
\to
  \vmS=\vaS\tilde\lambda
=\aS \sum_{k=0}^\infty kp_k (1 - \pi_k) =\mS-\aS \sum_{k=0}^\infty k \pi_k p_k.
\end{equation*}

These limits may be assumed almost sure by the Skorokhod coupling lemma.  Hence, we can
apply Theorems \ref{t:mI0} and \ref{t:mI0-1} to the
post-vaccination epidemic with the modified values $\vaS, \vaI, \vaR$,  $\tilde \lambda, \vmI, \vmR$ and $(\tilde p_k)_0^\infty$, with the understanding that $R_t$ and $\XR{t}$
now include vaccinated vertices and half-edges (though these could easily be subtracted off).
The limiting deterministic evolution and final size follow immediately.
Rather than restating the results in full, we simply give criteria for the
vaccination programme
to be successful.

\begin{corollary}
Let us consider the SIR epidemic on the uniform simple graph $\Gr$.
Suppose that \ref{d:first}--\ref{d:last} and \ref{d:sumsquaredi=On} hold.
Suppose further that each initially susceptible vertex of degree $k$ is vaccinated
with probability $\pi_k \in [0,1)$ independently of the others, 
and $\mI = 0$.  Let
\begin{equation}\label{e:vaccR0}
\tRzero \= \left(\frac{\beta}{\rho + \beta}\right)\left(\frac{\aS}{\mX}\right)\sum_{k=0}^\infty k(k-1) p_k (1-\pi_k).
\end{equation}

If\/ $\tRzero \le 1$, then the total number of susceptible vertices that get infected is $o_p(n)$.
If\/ $\tRzero > 1$, then there exists $\delta > 0$ such that at least $\delta n$ susceptibles get infected with probability bounded away zero.

The same result holds for the SIR epidemic on the multigraph $\GrCM$, even
without assumption \ref{d:sumsquaredi=On}.
\end{corollary}

\section{The time-changed epidemic}\label{s:fastepidemic}
In this and the next two sections, we consider the SIR epidemic on the random multigraph
$\GrCM$; in \refS{s:simple} we transfer the results to the simple random
graph $\Gr$.

A key step in the proof is to alter the speed of the process by multiplying each transition
rate by a constant depending on the current state. The constant is chosen so that
each free susceptible half-edge gets infected at unit rate (or, equivalently, so that the infection pressure on the population is 1).  Specifically, if there
are $\xI \ge 1$ free infective half-edges, and a total of $x$ free
half-edges of any type, we multiply all
transition rates out of such a state by $(x-1)/\beta\xI > 0$.  Thus each infective vertex
recovers at rate $\rho(x-1)/\beta\xI$, and each free infective
half-edge pairs off at rate $(x-1)/\xI$. This change of rates accelerates
the epidemic in its `slow' phases, when the number of free infective
half-edges is $o(n)$ (beginning and end of the epidemic).
Later, we will invert the time change to obtain the original process.

We use Greek letters ($\tau$ and $\sigma$) for the time variable of the altered process as a reminder of the rate modification. The notation for the numbers of half-edges and vertices of each type in the modified process follows that for the original process, except that we superscript each variable with a prime.
For example, $\XItt{\tau}$ denotes the number of infective half-edges at time $\tau \ge 0$
in the modified process.

Let
\begin{equation}\label{e:iEndtt}
 \iEndtt \= \inf\{\tau \ge 0  : \XItt{\tau} = 0 \}
\end{equation}
be the time at which the modified process stops, when there are no free infective half-edges.

\begin{lemma}\label{l:LLN}
Suppose that~\ref{d:alphaprops}--\ref{d:SI2ndmoment} hold.
Fix $\tau_1 > 0$.  Then, uniformly over $[0, \tau_1 \wedge \iEndtt]$,
\begin{align}
\label{e:SvttLLN}
\Svtt{\tau} / n &\pto \fSv(e^{-\tau}),\\
\label{e:XSttLLN}
\XStt{\tau} /n &\pto \fXS(e^{-\tau} ),\\
\label{e:XttLLN}
\Xtt{\tau}/n &\pto \fXX(e^{-\tau} ),\\
\label{e:XRttLLN}
\XRtt{\tau}/n &\pto \fXR(e^{-\tau} ),
\intertext{and, consequently,}
\label{e:XIttLLN}
\XItt{\tau}/n &\pto \fXI(e^{-\tau} ).
\end{align}
\end{lemma}

\begin{proof}
For each $k \in {\mathbb Z}^+$, let $\Svktt{\tau}$ denote the number of susceptible vertices with $k$ half-edges at time $\tau \ge 0$ (we omit the qualifier
$\tau \le \iEndtt$ throughout the proof; any occurrence of `$\tau$' is understood to mean `$\tau \wedge \iEndtt$').
Thus
$\Svtt{\tau} = \sum_{k = 0}^\infty  \Svktt{\tau}$ and
$\XStt{\tau} = \sum_{k = 0}^\infty k \Svktt{\tau}$ for each $\tau$. Also,
$\Svktt{0} = \nSk$, for each $k$.

For each $k$, the only jumps in $\Svktt{\tau}$
are decrements by $1$, and these occur when an infective half-edge pairs off and chooses
one of the 
half-edges belonging to a susceptible vertex of degree $k$.
Hence, with the modified transition rates,
\begin{align}
\dd\Svktt{\tau} &= -\beta \XItt{\tau} \parfrac{\Xtt{\tau} - 1}{\beta\XItt{\tau}} \parfrac{k\Svktt{\tau}}{\Xtt{\tau} - 1}d\tau + dM_{\mathrm{S},\tau}(k) \nonumber \\
& = -k\Svktt{\tau}\dd\tau + \dd M_{\mathrm{S},\tau}(k), \label{e:dSvttk}
\end{align}
where $(M_{\mathrm{S},\tau}(k))_{\tau \ge 0}$ is a martingale starting from
$M_{\mathrm{S},0}(k) = 0$
\cite[Proposition 1.7]{EthierKurtz}.

The differential notation in \eqref{e:dSvttk} means that
\begin{equation}
 \label{e:SvkIntForm}
 \Svktt{\tau} = \Svktt{0} -k \int_0^\tau \Svktt{\sigma}\dd\sigma + M_{\mathrm{S},\tau}(k).
\end{equation}
Since $\Svktt{0} = \nSk$, it follows that
\begin{equation}\label{e:SvkIntDev}
  \begin{split}
 \abs{\Svktt{\tau} - \nSk e^{-k\tau} }
& = \biggabs{\Svktt{\tau} - \nSk
\left(1 - k\int_0^\tau  e^{-k\sigma}\dd\sigma\right) } \\
& = \biggabs{
 \int_0^\tau k\Bigpar{-\Svktt{\sigma} + \nSk e^{-k\sigma} }\dd\sigma +
 {M_{\mathrm{S},\tau}(k)} }\\
 & \le k \int_0^\tau \Bigabs{\Svktt{\sigma} - \nSk e^{-k\sigma} }\dd\sigma +
 \abs{M_{\mathrm{S},\tau}(k)}.
  \end{split}
\end{equation}
Consequently, using Gronwall's inequality (for positive bounded functions) \cite[Appendix \S{1}]{RevuzYor},
\begin{align*}
  \sup_{\tau \le \tau_1} \abs{\Svktt{\tau} - \nSk e^{-k\tau} } &\le k \int_0^{\tau_1} \sup_{\tau \le \sigma} \abs{\Svktt{\tau} - \nSk e^{-k\tau} }\dd\sigma +\sup_{\tau \le \tau_1}\abs{M_{\mathrm{S},\tau}(k)} \\
  & \le e^{k\tau_1} \sup_{\tau \le \tau_1}\abs{M_{\mathrm{S},\tau}(k)},
\end{align*}
and it follows that
\begin{equation}
  \label{e:maj}
 \sup_{\tau \le \tau_1} \bigabs{\Svktt{\tau}/n - \aS p_k e^{-k\tau} } \le \abs{ \nSk/n - \aS p_k } + e^{k\tau_1} \sup_{\tau \le \tau_1}\abs{M_{\mathrm{S},\tau}(k)}/n.
\end{equation}
The first term on the right goes to zero as $n\to \infty$ by \ref{d:alphaprops} and \ref{d:asympsuscdist}.  Let us show that $\sup_{\tau \le \tau_1}\abs{M_{\mathrm{S},\tau}(k)}/n \pto 0$.

The martingale $M_{\mathrm{S},\tau}(k)$ is
right continuous and has left limits (c\`adl\`ag), and it is also of finite variation.
The quadratic variation process of such a martingale is the running sum of its
(countably many) squared jumps \cite[Theorem 26.6]{Kallenberg}.
The jumps in $M_{\mathrm{S},\tau}(k)$ are by \eqref{e:SvkIntForm} the same
as the jumps in $\Svktt{\tau}$.
Each jump in $\Svktt{\tau}$ is a decrement by one, and there are at most $\Svktt{0}$ such jumps.  Hence
the quadratic variation $[M_{\mathrm{S}}(k)]_\tau$ of $M_{\mathrm{S},\tau}(k)$ satisfies
\[
 [M_{\mathrm{S}}(k)]_\tau = \sum_{0 \le \sigma \le \tau} (\Delta M_{\mathrm{S},\sigma}(k))^2 \le \Svktt{0} =
 \nSk
\le n,
\]
for any $\tau \ge 0$.
In particular, $\E[M_{\mathrm{S}}(k)]_{\tau}<\infty$
for every $\tau\ge0$,
so $M_{\mathrm{S},\tau}(k)$ is square integrable
and
$\E M_{\mathrm{S},\tau}(k)^2 = \E[M_{\mathrm{S}}(k)]_{\tau}$
\cite[Corollary 3 after Theorem II.6.27, p.~73]{Protter}.
Hence,
 Doob's $L^2$-inequality \cite[Proposition 7.16]{Kallenberg} yields
\begin{equation}\label{e:SvkDoobIneq}
 \E \sup_{\tau \le \tau_1} \abs{M_{\mathrm{S},\tau}(k)}^2
\le 4\E M_{\mathrm{S},\tau_1}(k)^2
= 4\E[M_{\mathrm{S}}(k)]_{\tau_1} = O(n).
\end{equation}
It follows that $\sup_{\tau \le \tau_1} \abs{M_{\mathrm{S},\tau}(k)} = \op(n)$, and so, by \eqref{e:maj}, for each $k$,
\begin{equation}\label{e:SvkttLLN}
 \sup_{\tau \le \tau_1} \bigabs{\Svktt{\tau}/n - \aS p_k e^{-k\tau} } \pto 0.
\end{equation}


Let $\eps  > 0$.  By
assumptions \ref{d:alphaprops}--\ref{d:suscmeanUS},
there exists $K > 0$ such that $\sum_{k = K+1}^\infty k \nSk /n < \epsilon$
for any $n$, see \refR{r:UI}.  Further, $K$ can be chosen large enough that $\sum_{k = K+1}^\infty k p_k < \eps$.
Consequently,
\begin{equation}\label{e:XSttUIbd}
  \begin{split}
\sup_{\tau \le \tau_1} \bigabs{ \XStt{\tau} /n - \fXS(e^{-\tau}) }
& = \sup_{\tau \le \tau_1} \lrabs{ \sum_{k=0}^\infty k \Svktt{\tau} /n - \aS \sum_{k=0}^\infty k p_k e^{-\tau k} } \\
& \le \sum_{k =0}^K k \sup_{\tau \le \tau_1} \left| \Svktt{\tau} /n - \aS
p_k e^{-k\tau}  \right|
+ \sum_{k = K+1}^\infty k\left(\nSk/n + p_k\right)\\
& \le \sum_{k =0}^K k \sup_{\tau \le \tau_1} \left| \Svktt{\tau} /n - \aS
p_k e^{-k\tau}  \right| + 2\epsilon,
  \end{split}
\end{equation}
and the same bound applies to $\sup_{\tau \le \tau_1} \abs{ \Svtt{\tau} /n - \fSv(e^{-\tau}) }$.
The finite sum in the last line of \eqref{e:XSttUIbd} tends to zero in probability, by \eqref{e:SvkttLLN}.  This completes our proof of \eqref{e:SvttLLN}
and \eqref{e:XSttLLN}.

We prove~\eqref{e:XttLLN} and \eqref{e:XRttLLN} similarly.  The total
number of free half-edges decrements by $2$ whenever an infective
half-edge pairs off.  Then
\begin{equation}
  \begin{split}
 d\Xtt{\tau} &
= -2\beta \XItt{\tau} \parfrac{\Xtt{\tau} - 1}{\beta\XItt{\tau}}d\tau
+ dM_{\mathrm{X},\tau}
= -2 (\Xtt{\tau} - 1)\dd\tau + dM_{\mathrm{X},\tau},\label{e:dXtt}	
  \end{split}
\end{equation}
where $(M_{\mathrm{X},\tau})_{\tau \ge 0}$ is a c\`adl\`ag martingale
satisfying $M_{\mathrm{X},0} = 0$.
Writing the equation in integrated form
and proceeding as in \eqref{e:SvkIntDev},
\[
 \bigabs{\Xtt{\tau} - \Xtt{0} e^{-2 \tau} } \le 2 \int_0^\tau  \bigabs{\Xtt{\sigma} - \Xtt{0} e^{-2 \sigma} }\dd\sigma + 2\tau + \abs{M_{\mathrm{X},\tau}}.
\]
Gronwall's inequality yields
\[
 \sup_{\tau \le \tau_1}\bigabs{\Xtt{\tau} - \Xtt{0} e^{-2 \tau} } \le e^{2 \tau_1} \left(2 \tau_1 + \sup_{\tau \le \tau_1}\abs{M_{\mathrm{X},\tau}}\right),
\]
and so, 
\[
 \sup_{\tau \le \tau_1} \bigabs{\Xtt{\tau}/n - \fXX(e^{-\tau}) } \le \bigabs{\Xtt{0}/n - \mX} + e^{2 \tau_1} \left(2\tau_1 + \sup_{\tau \le \tau_1}\abs{M_{\mathrm{X},\tau}}\right)/n.
\]

By \ref{d:meanIRconv} $\Xtt{0}/n = \sum_{k=0}^\infty k \nk/n \to \mX$, and so the first term above converges to 0.
To estimate the martingale term, note that the jumps in $M_{\mathrm{X},\tau}$ are the same as the jumps in $\Xtt{\tau}$. At each jump, $\Xtt{\tau}$ decreases by $2$, and there is at most one jump in $\Xtt{\tau}$ for
each of the $\Xtt{0}$ initial half-edges.  Hence, the quadratic variation $[M_{\mathrm{X}}]_{\tau_1}$ satisfies
\[
 [M_{\mathrm{X}}]_{\tau_1} \le \sum_{\sigma \ge 0} (\Delta M_{\mathrm{X},\sigma})^2 \le 4 \Xtt{0} = O(n).
\]
Proceeding as in \eqref{e:SvkDoobIneq}, we see that $\sup_{\tau \le \tau_1}\abs{M_{\mathrm{X},\tau}} = \op(n)$, and \eqref{e:XttLLN} follows.

Now, the number of free recovered half-edges $\XRtt{\tau}$
decreases by 1 when an infective half-edge is paired with a recovered half-edge, and increases by $k \ge 0$ when an infective
vertex with $k$ free half-edges recovers.  With the modified rates, we have
\begin{align}
d\XRtt{\tau} &= \left(-\beta \XItt{\tau} \parfrac{\XRtt{\tau}}{\Xtt{\tau} - 1}  + \rho \XItt{\tau}\right)\parfrac{\Xtt{\tau} - 1}{\beta\XItt{\tau}}d\tau + dM_{\mathrm{R},\tau}\nonumber\\
 & = -\XRtt{\tau}\dd\tau + \rho \beta^{-1} (\Xtt{\tau} - 1)\dd\tau  + dM_{\mathrm{R},\tau},\label{e:dXRtt}
\end{align}
where $(M_{\mathrm{R},\tau})_{\tau \ge 0}$ is a c\`adl\`ag martingale with
$M_{\mathrm{R},0} = 0$.

Differentiating the expression for $\fXR$ in \eqref{e:fXR} shows that
\[
 \frac{d}{d\sigma} \fXR(e^{-\sigma}) = -\fXR(e^{-\sigma}) + \rho\beta^{-1}\fXX(e^{-\sigma}).
\]
Since $\fXR(1) = \mR$, integrating, subtracting the integrated form
of \eqref{e:dXRtt} divided by $n$,
and taking the absolute value yields
\begin{equation}\label{e:XRttIntForm}
 \bigabs{\XRtt{\tau}/n - \fXR(e^{-\tau})} \le \int_0^\tau  \bigabs{\XRtt{\sigma}/n - \fXR(e^{-\sigma})}\dd\sigma + E_\tau,
\end{equation}
where the absolute error term $E_\tau$ is given by
\[
  E_\tau \= \bigabs{\XRtt{0}/n -\mR} + \frac{\rho}{\beta}\int_0^\tau  \bigabs{\Xtt{\sigma}/n - \fXX(e^{-\sigma})}\dd\sigma + \frac{\rho\tau}{n \beta} + \frac{\abs{M_{\mathrm{R},\tau}}}{n}.
\]

Let us show that $\sup_{\tau \le \tau_1} \abs{E_\tau} \pto 0$; then~\eqref{e:XRttLLN} follows from~\eqref{e:XRttIntForm} by applying Gronwall's inequality.

First of all, $\XRtt{0}/n = \sum_{k = 0}^\infty k \nRk/n \to \mR$ by \ref{d:meanIRconv}.
The integrand in the second term tends to zero uniformly by the convergence \eqref{e:XttLLN} of $\Xtt{\sigma}$ already proven.
Finally, the jumps in $M_{\mathrm{R},\tau}$ are due to either an infective vertex recovering
(which happens at most once for each vertex, and only for vertices that were initially infective or initially susceptible) or an infective half-edge
pairing to a recovered half-edge (which happens at most once for each half-edge). It follows that
\[
[M_{\mathrm{R}}]_{\tau_1}  \le
\sum_{\sigma \ge 0}(\Delta M_{\mathrm{R},\sigma})^2
 \le X_0 + \sum_{k=0}^\infty k^2 (\nSk + \nIk) = o(n^2),
\]
by \ref{d:SI2ndmoment} and \eqref{e:SI2ndmoment},
and so $\sup_{\tau \le \tau_1} \abs{M_{\mathrm{R},\tau}} = \op(n)$.

Finally, the convergence \eqref{e:XIttLLN} of $\XItt{\tau}/n = \Xtt{\tau}/n - \XStt{\tau}/n - \XRtt{\tau}/n$ follows by applying the triangle inequality.
\end{proof}

\begin{remark}
Any given susceptible half-edge gets infected at unit rate in the modified process, and so
each initial degree $k$ susceptible gets infected after an $\Exp(k)$ time, independently of
all the other susceptibles.  This observation can be used to give an alternative proof of \eqref{e:SvkttLLN} 
using the Glivenko--Cantelli lemma for convergence of empirical distributions, as in \cite{JansonLuczakGiantCpt}.  Using the Glivenko--Cantelli lemma would allow us to take
$\tau_1 = \infty$ in~\eqref{e:SvttLLN}--\eqref{e:XttLLN}.
However, this strengthening would give no extra benefit, since $\iEndtt$ is bounded \whp, as shown in \refS{s:proofs} below.  

Further, it is possible to obtain quantitative statements of convergence for
$\sup_{\tau \le \tau_1}\abs{M_{\mathrm{S},\tau}}/n$, $\sup_{\tau \le \tau_1}\abs{M_{\mathrm{X},\tau}}/n$ and $\sup_{\tau \le \tau_1}\abs{M_{\mathrm{R},\tau}}/n$
using techniques such as those in~\cite{NorrisDarling}.
\end{remark}

\subsection{Inverting the rate change to recover the original process}\label{s:invTT}

To close this section, we explain how to rescale time in order to obtain the original process.
To this end, we define 
\begin{equation}\label{e:AddFunc}
 A_\tau = \int_0^\tau \frac{1}{\beta} \parfrac{\Xtt{\sigma} - 1}{\XItt{\sigma}}d\sigma, \qquad \tau \ge 0,
\end{equation}
where we regard the bracketed term in the integrand as being equal to $1/2$
if $\XItt{\sigma} = 0$, \ie{} if $\sigma \ge \iEndtt$
 (till then, the bracketed term is at least $1/2$, since $\Xtt{\sigma} \ge
\XItt{\sigma}$, and, if $\XItt{\sigma} = 1$, then $\Xtt{\sigma} \ge 2$).
Thus $A$ is strictly increasing and continuous.  We let $\tau(t):[0,\infty) \to [0, \infty)$ be
its strictly increasing continuous inverse, so $A_{\tau(t)} = t$ for any $t \ge 0$.

The processes in their original time scaling can then be realised as
\begin{align}\label{e:realtimeprocs}
\Sv{t} = \Svtt{\tau(t)},&&
\Iv{t} = \Ivtt{\tau(t)}, &&
\Rv{t} = \Rvtt{\tau(t)},&&
\ldots,  &&&
t \ge 0,
\end{align}
by applying \refL{l:AddFuncTT} to the underlying Markov evolution of the
epidemic and graph dynamics.  Since the epidemic stops when we run out of free infective half-edges, it makes no difference to~(\ref{e:realtimeprocs})
if we replace $\tau(t)$ by
$\xtau(t) \= \tau(t)\wedge\iEndtt$ above; 
this will be convenient for the proofs below.

\section{Proofs for multigraphs: Theorems \ref{t:mIg0} and  \ref{t:mI0}}\label{s:proofs}

We continue to study the SIR epidemic  on the random multigraph $\GrCM$.
We assume  \ref{d:first}--\ref{d:last}, unless otherwise stated.
For simplicity,
we also assume \ref{d:sumsquaredi=On}, and leave the minor modifications
in the general case to \refSS{ss:nosecondmoment}.
(We are mainly interested in the simple random
graph $\Gr$, where we have to assume \ref{d:sumsquaredi=On}.)

We begin with the (sub)critical regime (part \ref{t:mI0.i} of \refT{t:mI0}).

\subsection{Subcriticality: proof of \refT{t:mI0} part \ref{t:mI0.i}}

Suppose that the hypotheses $\Rzero \le 1$ and $\mI = 0$ are satisfied. We must show
only $\op(n)$ susceptibles ever get infected, and  it is sufficient to prove this for the modified epidemic studied in \refS{s:fastepidemic}.
The key step is proving that the modified epidemic dies almost instantly, \ie{} $\iEndtt \pto 0$.

For this purpose, we first show that $\fXI(\theta) < 0$ for $\theta \in (0,1)$.  It is enough to consider
\begin{equation}\label{e:fXIthetaOvertheta}
h(\theta) \= \fXI(\theta)/\theta = \mX \theta -\aS \sum_{k = 0}^\infty k \theta^{k-1} p_k - \mR - \rho \mX(1- \theta)/\beta.
\end{equation}
By assumption, $h(1) = \fXI(1) = \mI = 0$, see \eqref{e:fX1}.
Furthermore,
\begin{equation}
 \label{e:h0<0}
 h(0) = - \aS p_1 - \mR - \rho \mX/\beta < 0,
\end{equation}
by \ref{d:p1orrhopos}.

Differentiating $h$ and substituting the identity
$\beta \aS \sum_{k=0}^\infty k(k-1)p_k = \mX(\rho + \beta)\Rzero$ yields
\begin{equation}
  \label{e:sture}
  \begin{split}
\beta h'(\theta) & = (\rho + \beta) \mX - \beta \aS \sum_{k = 0}^\infty k(k-1) \theta^{k-2} p_k \\
 & \ge (\rho + \beta) \mX ( 1 - \Rzero) \\
 & \ge 0,
  \end{split}
\end{equation}
for $\theta\in(0,1)$.
If $\sum_{k = 3}^\infty p_k > 0$ and $\theta  < 1$, then there is
strict inequality going from the first to second line in \eqref{e:sture},
and thus $h'(\theta)>0$.
If $\sum_{k = 3}^\infty p_k = 0$,
then $h$ is linear and  $h(0)<0$ by \eqref{e:h0<0}.
In any case,
recalling  that $h(1)=0$, we obtain
$h(\theta) < 0$ for $\theta \in(0,1)$.

We now take $\eps > 0$ and apply \refL{l:LLN} with $\tau_1 = \eps$.
It follows from \eqref{e:XIttLLN} that
\begin{equation}\label{e:earlyXIbd}
\sup_{\tau \le \iEndtt \wedge \eps} \bigabs{ \XItt{\tau}/n - \fXI(e^{-\tau}) } < \abs{\fXI(e^{-\eps})}/2
\end{equation}
\whpx.  On the event that inequality~(\ref{e:earlyXIbd}) holds, we have $\iEndtt < \eps$; otherwise the left hand side of \eqref{e:earlyXIbd} is at least $\abs{\fXI(e^{-\eps})}$,
since $\XItt{\eps} \ge 0$ and $\fXI(e^{-\eps})<  0$.
Hence \whp{} $\iEndtt<\eps$.

It follows that $\iEndtt \pto 0$, and,
by \eqref{e:SvttLLN}, the number $\nS - \Svtt{\iEndtt}$ of susceptibles that ever get infected
satisfies
\begin{equation}\label{e:Svsmall}
 (\nS - \Svtt{\iEndtt})/n \pto \aS  - \fSv(e^{-0}) = 0. 
\end{equation}

\begin{remark} \label{r:h=0}
 If \ref{d:p1orrhopos} does not hold, then $h(0) = 0=h(1)$.
By \eqref{e:sture}, then $h(\theta)=0$ for every $\theta$. This happens only
in the case $p_0+p_2=1$. (In this case \eqref{e:R0} yields $\Rzero = 1$.)
\refL{l:LLN} is still valid, but it is no longer true in general that
$\iEndtt\pto0$. For example, if all vertices have degree 2 and there is
initially a single infective vertex, then there will be two free infective
half-edges until one pairs off with the other, and it is easy to see that
$\iEndtt$ has an $\Exp(1)$ distribution. Furthermore, $\fSv(\theta)=\theta^2$
and it follows from \eqref{e:XSttLLN} that
$\Svtt{\iEndtt} / n \dto \fSv(e^{-\iEndtt})=e^{-2\iEndtt}\sim B(\frac12,1)$
(a beta distribution with density $\frac12x^{-1/2}$ on $\oi$),
so in this case there is a non-degenerate limiting distribution of the size
of the epidemic.

Recall that this example was considered at the end of \refS{s:results}.
It is easily seen that for the two modifications considered there, with
some vertices of degree 1 or 4,  we have $\iEndtt\pto0$
and $\iEndtt\pto\infty$, respectively, and thus
$\Svtt{\iEndtt} / n \pto 1$ and $\Svtt{\iEndtt} / n \pto 0$.
\end{remark}

\subsection{Many initially infective half-edges: proof of \refT{t:mIg0}}
\label{ss:pfmIg0}

We now consider the case where there is initially a large
number of infective half-edges, \ie{} $\mI > 0$.

\medskip

\noindent (a)
It suffices to prove that $h$ defined in \eqref{e:fXIthetaOvertheta} has a unique root $\fXIroot \in (0,1)$,
since $h(1) = \mI > 0$ and $h(0) = - \aS p_1 - \mR - \rho \mX/\beta < 0$  by
\eqref{e:h0<0} and \ref{d:p1orrhopos}.
Calculating $h''(\theta) = -\aS \sum_{k=0}^\infty k(k-1)(k-2)\theta^{k - 3}
p_k \le 0$ shows that $h$ is concave on $(0,1)$.
This, together with the inequality $h(1)>0$,
implies uniqueness of $\fXIroot$.

\medskip

\noindent (b)
By \ref{d:sumsquaredi=On},
$\sum_{k=0}^\infty k^2 p_k < \infty$,
see \refR{r:2ndmoment}, and by \eqref{e:fXS}--\eqref{e:fXI},
the derivative of $\fXI$ is bounded on
$[0,1]$.  Hence, $\pI$
is Lipschitz continuous on $[\fXIroot, 1]$.  Consequently, both existence and
uniqueness of the solution $\lTP{t}$, $t \ge 0$, to \eqref{e:dlTPtmIg0}
follow from standard theory. Note that the constant function $\fXIroot$ is
another solution to the differential equation, so $\lTP{t}>\fXIroot$ for all
$t$ by uniqueness of solutions.
Thus $\lTP{t}$ is strictly decreasing and bounded below,
so the limit $\lim_{\ttoo}\lTP{t}$ exists, and must be a zero of $\pI$, \ie,
by part \ref{t:mIg0.a},
the limit equals $\fXIroot$.

In the proof below we will use a more explicit form of the solution.
Let $\fpT \= -\ln\fXIroot$, and define $\lttInv{}:[0,\fpT) \to [0,\infty)$ by
\begin{equation}\label{e:lttInv}
 \lttInv{\tau} \= \int_0^\tau \frac{d\sigma} {\beta \pI(e^{-\sigma}) }, \qquad 0\le \tau < \fpT.
\end{equation}
The integrand is strictly positive on $[0, \fpT)$ and so $\lttInv{}$ is strictly increasing.
Furthermore, $\pI(e^{-\fpT}) = 0$ and $\pI'(e^{-\sigma})$ is bounded
for $\sigma$ in a neighbourhood of $\fpT$.
Hence $\pI(e^{-\sigma}) = O(\fpT - \sigma)$ for $\sigma \in [0, \fpT]$.
It follows that $\lttInv{\tau} \upto \infty$ as $\tau \upto \fpT$.
The inverse $\ltt:[0,\infty) \to [0,\fpT)$ of $\lttInv{}$ is strictly increasing and continuously
differentiable, and satisfies $\ltt'(t) = \beta \pI(\exp(-\ltt(t))$ by the Inverse Function Theorem, and $\ltt(0) = 0$.
So  $\lTP{t} = \exp(-\ltt(t))$ solves~\eqref{e:dlTPtmIg0}.

\medskip

\noindent (c)
We first show that $\iEndtt \pto \fpT$.  
Let us take a sufficiently small $\eps > 0$ and define
\begin{equation}
  \label{e:delta}
\delta
\= \inf_{\tau \le \fpT - \eps} \fXI(e^{-\tau}) \wedge |\fXI(e^{-(\fpT + \eps)})|.
\end{equation}
By part \ref{t:mIg0.a}, $\delta > 0$.
Then \refL{l:LLN} (with $\tau_1 = \fpT + \eps$) 
shows that \whp
\begin{equation}\label{e:XIttLLNA}
\sup_{\tau \le \iEndtt \wedge (\fpT + \eps)} \bigabs{\XItt{\tau}/n - \fXI(e^{-\tau})} < \delta/2.
\end{equation}
We claim that, on that event, $\fpT - \eps < \iEndtt < \fpT + \eps$.
Indeed, if $\iEndtt \le \fpT - \eps$, then the left hand side of \eqref{e:XIttLLNA}
it is at least $\abs{\XItt{\iEndtt}/n - \fXI(e^{-\iEndtt})} = \fXI(e^{-\iEndtt}) \ge \delta$,
by definition of $\delta$ and the fact that $\XItt{\iEndtt} = 0$.  If $\iEndtt \ge \fpT + \eps$,  then the left hand side of \eqref{e:XIttLLNA}
is at least $\abs{\XItt{\fpT + \eps}/n - \fXI(e^{-(\fpT + \eps)})} \ge \abs{\fXI(e^{-(\fpT + \eps)})} \ge \delta$,
by definition of $\delta$ and the fact that $\XItt{\fpT + \eps} \ge 0$ and
$\fXI(e^{-(\fpT + \eps)}) < 0$.
Thus, \whp{}
\begin{equation}
  \label{jeppe}
\fpT - \eps < \iEndtt < \fpT + \eps.
\end{equation}

Our next task is to return to the original process, and, to this end, we need to study
the process $(A_\tau)_{\tau \ge 0}$
defined in
\eqref{e:AddFunc} and its inverse $\tau (t)$.  The integrand in \eqref{e:AddFunc} converges to $1/(\beta \pI(e^{-\sigma}))$ uniformly in probability on $[0,\fpT - \eps]$ by
\eqref{e:XttLLN}, \eqref{e:XIttLLN}, and the fact that $\fXI(e^{-\sigma}) \ge \delta > 0$ for $0 \le \sigma \le  \fpT - \eps$.
Therefore,
\begin{equation}\label{e:AddFuncConv}
A_\tau \pto \lttInv{\tau}
\end{equation}
uniformly over $0 \le \tau \le \fpT - \eps$, where $\lttInv{\tau}$ is as in~(\ref{e:lttInv}).

Let $t_1 \= \lttInv{\fpT - 2\eps}$.
(We assume $\eps<\fpT/2$.)
The uniform convergence \eqref{e:AddFuncConv} 
and strict monotonicity of $\lttInv{}$ imply that $A_{\fpT - \eps} > t_1$
\whpx.
On this event, $\tau(t) \le \tau(t_1)   < \fpT - \eps$ for $t \le t_1$,
and so \eqref{e:AddFuncConv} shows that
\begin{equation}\label{e:AddFuncConv1}
t-  \lttInv{\tau(t)}
=
 A_{\tau(t)} - \lttInv{\tau(t)} \pto 0,
\end{equation}
uniformly on $t \le t_1$.  Recall from the proof of (b) that $\hat{\tau}(t)$ is the inverse of the function $\lttInv{\tau}$; then $0 \le \ltt'(t) = \beta \pI(\exp(-\ltt(t))\le \beta$,
and so $\ltt$ is uniformly continuous.  Hence, \eqref{e:AddFuncConv1}  implies
\begin{equation}
  \label{abba}
  \sup_{t \le t_1}\bigabs{\ltt(t) - \tau(t)} \pto 0.
\end{equation}
Recall the definition $\xtau(t)\=\tau(t)\wedge\iEndtt$.
If $t \ge t_1$, then \whp{}, using \eqref{abba} and \eqref{jeppe},
\[
\fpT - 3 \eps =\ltt(t_1)-\eps< \xtau(t_1) \le \xtau(t) \le \iEndtt < \fpT + \eps,
\]
from which it follows that \whp
\begin{equation}
  \label{baab}
 \sup_{t \ge t_1} \bigabs{\ltt(t) - \xtau(t)} < 
3\eps.
\end{equation}
We conclude from \eqref{abba}--\eqref{baab}
that $\xtau(t) \pto \ltt(t)$, and hence that $\exp(-\xtau(t)) \pto \lTP{t} = \exp(-\ltt(t))$, uniformly over all $t \ge 0$.
It now follows that
\begin{align*}
\sup_{t \ge 0}\abs{\Sv{t}/n - \fSv(\lTP{t})} & = \sup_{t \ge 0}\abs{\Svtt{\xtau(t)}/n - \fSv(\lTP{t})} \\
& \le
\sup_{t \ge 0}\abs{\Svtt{\xtau(t)}/n - \fSv(e^{-\xtau(t)})}
+ \sup_{t \ge 0}\abs{\fSv(e^{-\xtau(t)}) - \fSv(\lTP{t})}
\end{align*}
converges to zero in probability, by \eqref{e:SvttLLN} and
the uniform continuity of $\fSv$ on $[0,1]$.  The
convergence statements in \eqref{e:convmIg0X}
follow similarly from
\refL{l:LLN} and uniform continuity of the functions  
$\fXS$, $\fXI$, $\fXR$. The convergence of $X_t/n$ also follows, as $X_t = \XS{t} + \XI{t} + \XR{t}$.

Since $\Rv{t}/n=1-\Sv{t}/n-\Iv{t}/n$, it remains to prove convergence for the fraction $I_t/n$ of infective vertices  in~\eqref{e:convmIg0SIR}.  We will work directly with the original process, using a compactness argument.
The number of infective vertices $I_t$ increases by 1 when a free
infective half-edge
is paired with a free suceptible half-edge, and decreases by 1 when an
infective vertex recovers. Therefore,
\begin{equation}\label{e:dIvt}
d\Iv{t} = \beta \XI{t} \parfrac{\XS{t}}{\X{t} - 1}dt - \rho \Iv{t}dt + dM_{\mathrm{I},t},
\end{equation}
where $(M_{\mathrm{I},t})_{t \ge 0}$ is a c\`adl\`ag martingale with
$M_{\mathrm{I},0} = 0$
and quadratic variation
\[
 [M_{\mathrm{I}}]_t \le \sum_{s \ge 0} (\Delta I_s)^2
\le 2n,
\]
since each vertex can only get infected or recover at most once.
As in the proof of \refL{l:LLN}, Doob's inequality then gives
$\sup_{t \ge 0} \abs{M_{\mathrm{I},t}} = \op(n)$.

Writing \eqref{e:dIvt} in integral form and dividing by $n$, we obtain
\begin{equation}\label{e:intIvt}
(\Iv{t} - \Iv{0} -M_{\mathrm{I},t})/n = \int_0^t \left(\parfrac{\beta \XI{s}\XS{s}}{n(\X{s} - 1)} - \frac{\rho \Iv{s}}{n} \right)ds,
\end{equation}
and the integrand is bounded in absolute value (by $2\beta\mu + \rho$ for $n$ large).  Hence, $(\Iv{t} - \Iv{0} -M_{\mathrm{I},t})/n$, $n \ge 0$,
is a uniformly Lipschitz family, and
it is also uniformly bounded on each finite interval $[0,t_1]$.
Thus, the Arzela--Ascoli theorem implies that it is tight in $C[0,t_1]$ for
any $t_1> 0$ \cite[Theorems A2.1 and  16.5]{Kallenberg},
and so also in $C[0,\infty)$.
This then implies that there exists a subsequence along which the process converges in distribution in $C[0,\infty) \subset D[0,\infty)$,
and the same holds for $(\Iv{t} - \Iv{0})/n$ in $D[0,\infty)$ since $\sup_{t \ge 0} \abs{M_{\mathrm{I},t}} = \op(n)$.  The convergence may
be assumed almost sure by the Skorokhod coupling lemma.

Hence, along the subsequence, $\Iv{t}/n$ \as{}
converges, uniformly on compact sets,
to a continuous limit $\lIv{t}$. Clearly, $\lIv{0} = \aI$.
Let us show that $\lIv{t}$ is deterministic. Since ${M_{\mathrm{I},t}} = \op(n)$,
\eqref{e:intIvt}, the dominated convergence theorem,
and the uniform convergence \eqref{e:convmIg0X} of $\X{t}/n$, $\XI{t}/n$ and $\XS{t}/n$ (which may also be assumed \as{}), together imply that
\[
  \lIv{t}  = \lim_{n \to \infty} \Iv{t}/n  = \aI + \int_0^t\left( \parfrac{\beta \fXI(\lTP{s})\fXS(\lTP{s}) }{\fXX(\lTP{s})} - \rho \lIv{s} \right)ds.
\]
Consequently,  $\lIv{t}$ is continuously differentiable, and,
differentiating, we obtain the
differential equation \eqref{e:dlIvtTmIg0} with the unique solution
\begin{equation}\label{e:lIvt}
  \lIv{t} =\aI e^{-\rho t} + \int_0^t e^{-\rho (t-s)} \parfrac{\beta \fXI(\lTP{s})\fXS(\lTP{s}) }{\fXX(\lTP{s})}ds.
\end{equation}
Hence, the subsequential limit $\lIv{t}$ is uniquely determined, which
implies convergence along
the original sequence, and the convergence may be assumed almost sure.  The limit is
continuous, so the convergence is uniform on $[0,t_1]$, for any $t_1  > 0$.

Let $t_1 > 0$ and let $t \ge t_1$.  The number of recovered vertices $\Rv{t}$ is increasing in $t$, so
\[
 0 \le \Iv{t}/n = 1 - \Sv{t}/n - \Rv{t}/n \le 1 - \Sv{t}/n - \Rv{t_1}/n.
\]
But $\Rv{t_1}/n \pto 1 - \lIv{t_1} - \fSv(\lTP{t_1})$ and $\Sv{t}/n \ge
\inf_{t \ge 0} \Sv{t}/n \pto \fSv(\fXIroot)$.
So for any $t_1>0$ and $\eps > 0$, \whp{} for all $t\ge t_1$,
\begin{equation}\label{e:bdItvt1}
 0 \le \Iv{t}/n \le \fSv(\lTP{t_1}) - \fSv(\fXIroot) + \lIv{t_1} + \eps,
\end{equation}
and
$\fSv(\lTP{t_1}) \to \fSv(\fXIroot)$
as $t_1 \to \infty$,
by continuity of $\fSv$.

In the case $\rho>0$, we make a change of variable $s=t-u$ in \eqref{e:lIvt}. Then $\fXI(\lTP{t-u})\to\fXI(\fXIroot)=0$ as \ttoo{}, and, together with the dominated convergence theorem, this shows that $\lIv{t}\to0$ as $\ttoo$.
Since
$|\Iv{t}/n-\lIv{t}|\le\max\set{\Iv{t}/n,\lIv{t}}$,
it follows from \eqref{e:bdItvt1} that \whp
\begin{equation}\label{e:bdItvt10}
\sup_{t\ge t_1} |\Iv{t}/n -\lIv{t}|
\le \fSv(\lTP{t_1}) - \fSv(\fXIroot) + \sup_{t\ge t_1}\lIv{t} + \eps,
\end{equation}
which is at most $2\eps$ if $t_1$ is large enough.

In the case $\rho=0$ (no recoveries), we instead note that $\Iv{t}$ and $\lIv{t}$ are
non-decreasing. Then, since $\lIv{t}\le1$, $\lIv\infty\=\lim_{\ttoo}\lIv{t}$ exists,
and \whp{}
\begin{equation}\label{e:bdItvt11}
\sup_{t\ge t_1} ( \lIv{t}-\Iv{t}/n) \le \lIv\infty-\Iv{t_1}/n
\le  \lIv\infty-\lIv{t_1}+\eps.
\end{equation}
Together with \eqref{e:bdItvt1}, this shows that
$\sup_{t\ge t_1} |\Iv{t}/n-\lIv{t}| \le2\eps$ \whp{} if $t_1$ is large enough.

In summary, in both cases, for any $\eps>0$ we can choose $t_1$ such
that
$\sup_{t\ge t_1} |\Iv{t}/n-\lIv{t}| \le2\eps$ \whpx. Since also
$\Iv{t}/n\pto\lIv{t}$ uniformly on $[0,t_1]$, for any $t_1 > 0$,
it follows that
$\Iv{t}/n \pto \lIv{t}$
uniformly over the whole of $[0,\infty)$. This completes the proof of~\eqref{e:convmIg0SIR}.

\medskip
\noindent (d)
This statement is an immediate consequence of \eqref{e:convmIg0SIR}.

\begin{remark}
If \ref{d:p1orrhopos} does not hold then $h(0) = 0$ and $h(1) = \mI>0$,
so $h(\theta) \ge \mI \theta > 0$ for
$\theta \in (0,1]$, by concavity of $h$.  So the only root of $\fXI(\theta)$
in $[0,1]$ is at zero,
and the argument in (c) above shows that,
for any fixed $\tau_1 > 0$, $\iEndtt > \tau_1$ \whpx.
Thus $\iEndtt\pto\infty$, and it follows from \eqref{e:SvttLLN} that
$\Svtt{\iEndtt}/n\pto\fSv(0)=\aS p_0$. Hence, apart from the isolated
vertices, all but $\op(n)$ of the susceptible vertices succumb to infection.
\end{remark}

\subsection{Supercriticality: proof of \refT{t:mI0} part \ref{t:mI0.ii}}
\label{ss:pfmI0}

We now consider the `supercritical' regime, where the basic reproductive ratio $\Rzero > 1$ and there are few initially infective vertices ($\mI
= 0$). Let us start by noting that $\fXI(1) = \mI = 0$, see \eqref{e:fX1}.

We will say that an event $\cE_n$ holds \whp{} conditional on event ${\mathcal D}_n$ if
$\Prob(\cE_n\mid {\mathcal D}_n) = 1 - o(1)$ as $n \to \infty$, or, equivalently, if
$\Prob(\cE_n^c \mid {\mathcal D}_n) = o(1)$ as $n \to \infty$.
If $\Prob ({\mathcal D}_n)$ is bounded away from 0 as $n \to \infty$, then
$\cE_n$ holds \whp{} conditional on ${\mathcal D}_n$ if and only if
$\Prob({\mathcal D}_n, \cE_n^c) = o(1)$.

\medskip
\noindent (a)
It is sufficient to show that $h$ defined in \eqref{e:fXIthetaOvertheta}
has a unique root in $(0,1)$.
Differentiating $h$ and then substituting $\beta \aS\sum_{k=0}^\infty
k(k-1)p_k = \mX(\rho + \beta)\Rzero$ yields,
\cf{} \eqref{e:sture},
\[
\beta h'(1) = \mX(\rho + \beta)(1 - \Rzero) <  0.
\]
Together with the fact that $h(1) = \fXI(1) = \mI = 0$, this shows that $h(\theta) > 0$ for all $\theta < 1$ close enough to 1.
Furthermore, $h(0) = - \aS p_1 - \mR - \rho \mX/\beta < 0$ by
\ref{d:p1orrhopos}, and so there is at least one root in $(0,1)$.
Uniqueness follows from concavity of $h$, as in the proof  of
\refT{t:mIg0}\ref{t:mIg0.a}.

\medskip

\noindent (b) As in the proof of \refT{t:mIg0},
existence and uniqueness of the solution to \eqref{e:dlTPt}
follow from Lipschitz continuity of $\pI$ and standard theory, but in the proof below we will use
a more explicit form of the solution.

Recall that there is an extra parameter $\calibS \in ( \fSv(\fXIroot), \fSv(1))$
used  in the initial condition of~\eqref{e:dlTPt}
to calibrate the time shift.
Let $\taz \= -\ln\fSv^{-1}(\calibS)$ and
$\fpT \= -\ln\fXIroot$; thus $0<\taz<\fpT$.
We define $\lttInv{}:(0,\fpT) \to \R$ by
\begin{equation}\label{e:lttInv2}
 \lttInv{\tau} \= \int_{\taz}^\tau \frac{d\sigma} {\beta \pI(e^{-\sigma}) },
\qquad 0 < \tau < \fpT.
\end{equation}
Then $\lttInv{}$ is strictly increasing, $\lttInv{\taz} = 0$,
and $\lttInv{\tau}\upto \infty$ as $\tau \upto \fpT$.
By \ref{d:sumsquaredi=On},
$\pI$  is Lipschitz continuous on $[\fXIroot, 1]$
and $\pI(e^{-\sigma})=O(\sigma)$, because $\pI(1)=0$;
hence $\lttInv{\tau} \downto -\infty$ as $\tau \downto 0$.
The inverse $\ltt:\R \to (0,\fpT)$ of $\lttInv{}$ is strictly increasing and continuously
differentiable. Also,
\begin{equation}
  \label{ltt'}
\ltt'(t) = \beta \pI(\exp(-\ltt(t))),
\end{equation}
and $\ltt(0) = \taz$. Hence, $\lTP{t} = \exp(-\ltt(t))$
satisfies \eqref{e:dlTPt}.
The final statements follow from
$\ltt(t)\upto \fpT = -\ln\fXIroot$ as $t\to\infty$ and
$\ltt(t)\downto 0$ as $t\to -\infty$.

\medskip
\noindent (c)
The proof of this part is deferred till~\refSS{s:large-epidemic}.

\medskip
\noindent (d)   Proceeding as in
the proof of \refT{t:mIg0}(c), but now defining
\begin{equation}
\delta
\= \inf_{\eps\le\tau \le \fpT - \eps} \fXI(e^{-\tau})
\wedge |\fXI(e^{-(\fpT + \eps)})|
\end{equation}
instead of \eqref{e:delta},
 we deduce that, for any $\eps > 0$, either $0 \le \iEndtt < \eps$ or $\fpT
 - \eps < \iEndtt < \fpT + \eps$, \whp; in other words,
the time-changed epidemic dies out either almost instantaneously or
around time $\fpT \=-\ln\fXIroot$.

Recall that $\iUp \= \inf\{ t \ge 0: \Sv{t} \le n\calibS \}$, and
assume that $\eps > 0$ is small enough that $\fSv(e^{-\eps}) > \calibS$.
By \eqref{e:SvttLLN},
\[
 \inf_{t \ge 0} \Sv{t}/n = \Svtt{\iEndtt}/n = \fSv(e^{-\iEndtt})+\op(1),
\]
and so $\Prob (\iEndtt < \eps, \iUp < \infty) =o(1)$.

For the rest of this proof, we will condition on the event $\iUp < \infty$.
Since by part (c), $\Prob (\iUp < \infty)$ is bounded away from 0 as $n \to
\infty$,
the previous two paragraphs show that, for sufficiently small $\eps >
0$, $\fpT - \eps < \iEndtt < \fpT + \eps$ holds \whp{}
conditional on $\iUp < \infty$.

As in the proof of~\refT{t:mIg0}, we need to study the inverse $\tau (t)$ of the process $A_\tau$
defined in~\eqref{e:AddFunc}, now with the added complication of
the time shift $\iUp$.
Here, we extend the definition of $\tau$ to $(-\infty, 0)$ by letting $\tau(t)\=0$ for $t<0$;
note that this agrees
with our convention to take
$\Sv{t} = \Sv{0}$ for $t < 0$, etc.,
and that
\eqref{e:realtimeprocs} holds for all $t\in(-\infty,\infty)$.

The calibration time $\iUp = \inf\{t\ge 0: \Svtt{\tau(t)} \le \calibS n\}$ satisfies
$\tau(\iUp) \pto \taz \= -\ln \fSv^{-1}(\calibS)$, by \eqref{e:SvttLLN}
and the fact that $\fSv(e^{-\tau})$ is strictly decreasing.
(We use part (c) to guarantee that
the convergence in probability survives the conditioning.
This is done without further comment below.)

By~\eqref{e:XttLLN},~\eqref{e:XIttLLN}, and the fact that $\inf_{\eps \le \sigma \le \fpT - \eps}\fXI(e^{-\sigma}) > 0$,
\begin{equation}\label{e:addFuncConvTT}
A_\tau - A_{\tau'} \pto \int_{\tau'}^\tau \frac{d\sigma}{\beta\pI(e^{-\sigma})}
= \lttInv{\tau} - \lttInv{\tau'},
\end{equation}
uniformly over $\tau, \tau' \in [\eps, \fpT - \eps]$.  Let us assume that $\eps$ is small enough that $\taz \in [2\eps, \fpT - 2\eps]$.
Taking $\tau' = \tau(\iUp)$ in~\eqref{e:addFuncConvTT} and using the fact that $\tau(\iUp)
\pto \taz$ gives
\begin{equation}\label{e:addFuncConvTTShift}
 A_\tau - \iUp
=
 A_\tau - A_{\tau(\iUp)} \pto  \lttInv{\tau} - \lttInv{\taz} = \lttInv{\tau},
\end{equation}
uniformly over $\tau \in [\eps, \fpT - \eps]$.

Now suppose that $t_1 > 0$ is given.  We may assume that $\eps$ is small enough that
$t_1 < (-\lttInv{\eps}) \wedge \lttInv{\fpT - \eps}$ (recall from part (b)
that the \rhs{} diverges as $\eps \to 0$).  Then
$A_\eps-\iUp\pto \lttInv{\eps}<-t_1$ by \eqref{e:addFuncConvTTShift}, so,
\whp{} conditional on $\iUp < \infty$, $A_\eps<\iUp-t_1$ and thus $\tau(\iUp-t_1)>\eps$.
(In particular, $\iUp-t_1>0$, \whp{} conditional on $\iUp < \infty$.
Since $t_1$ is arbitrary, this shows
$\iUp\pto\infty$.)

Similarly, \whp{} conditional on $\iUp < \infty$, $ A_{\fpT - \eps}-\iUp>t_1$ and
$\tau(\iUp+t_1)<\fpT - \eps$.
Thus,
\whp{} conditional on $\iUp < \infty$, $\tau(\iUp + t) \in [\eps, \fpT - \eps]$ for all
$t \in [-t_1,t_1]$.  So, taking $\tau = \tau(\iUp + t)$ in
\eqref{e:addFuncConvTTShift} shows that $\lttInv{\tau(\iUp + t)} \pto t$ uniformly on $t \in [-t_1,t_1]$.
Thus $\tau(\iUp + t) \pto \ltt(t)$ uniformly over $t \in [-t_1, t_1]$ by
uniform continuity of $\ltt(t)$, and this holds for any $t_1  >0$.
Furthermore, as shown above, \whp{} conditional on $\iUp < \infty$,
$\tau(\iUp + t) < \fpT - \eps<\iEndtt$ for all $t \in [-t_1,t_1]$,
and thus
$\xtau(\iUp + t) \=\tau(\iUp + t)\wedge\iEndtt=\tau(\iUp+t)$ for $t\le t_1$.
Hence, 
\begin{equation}\label{e:xtau0}
\xtau(\iUp + t)
\pto \ltt(t),
\end{equation}
uniformly over $t
\in [-t_1,t_1]$.
As before,
uniform convergence for $|t| \ge t_1$ follows  by monotonicity: if $t \ge
t_1$ then \whp{} conditional on $\iUp < \infty$, by \eqref{e:xtau0},
$0\le \xtau(\iUp - t) \le \xtau(\iUp-t_1)< \ltt(-t_1) + \eps$ and
$\ltt(t_1) - \eps <\xtau(\iUp+t_1)\le \xtau(\iUp + t) \le \iEndtt<\fpT +
\eps$.
Hence, \whp{} conditional on $\iUp < \infty$,
$\abs{\xtau(\iUp-t)-\ltt(-t)}\le \max\{\xtau(\iUp-t),\ltt(-t)\}\le \ltt(-t_1)+\eps$
and $\abs{\xtau(\iUp+t)-\ltt(t)}\le \fpT-\ltt(t_1)+\eps$ for all $t\ge t_1$,
and the right hand sides can be made arbitrarily small by choosing $t_1$
large and $\eps$ small. Hence, 
\eqref{e:xtau0} holds uniformly on $\R$.

Consequently, $\exp(-\xtau(\iUp + t)) \pto \exp(-\ltt(t))=\lTP{t}$ uniformly
on $\R$.  The convergence
$\Sv{\iUp+t}/n \pto \fSv(\lTP{t})$
in \eqref{e:convscSIR} and
the limits in \eqref{e:convscX}
then follow from \refL{l:LLN} and
uniform continuity of the limit functions, just as in the proof of
\refT{t:mIg0}.

The compactness argument from \refSS{ss:pfmIg0} shows that $\Iv{\iUp + t}/n$ converges
in distribution in $D(-\infty, \infty)$ along a subsequence to a differentiable
process $\lIv{t}$ that satisfies
\begin{equation}\label{e:dlIvtmI0}
\frac{d}{dt} \lIv{t} =   \beta\frac{\fXI(\lTP{t})\fXS(\lTP{t}) }{\fXX(\lTP{t})} - \rho \lIv{t}, \qquad t \in \R.
\end{equation}
Recovered vertices never become infective, and so, by~\eqref{e:convscSIR} and the fact that $\nI/n \to \aI = 0$,
\begin{equation}\label{e:IvttoaI}
 0 \le \Iv{\iUp + t}/n \le \nI/n + (\nS - \Sv{\iUp + t} )/n
\pto \aS - \fSv(\lTP{t}).
\end{equation}
Hence, $\lIv{t} \le \aS - \fSv(\lTP{t})$,
which tends to zero as $t \to -\infty$
and hence $\lIv{t} \to 0$ as well.  It
follows that the subsequential limit $\lIv{t}$ is unique and deterministic, given by
\begin{equation}\label{e:dlIvt}
 \lIv{t} = \int_{-\infty}^t e^{-\rho (t-s)} \,
\frac{\beta \fXI(\lTP{s})\fXS(\lTP{s}) }{\fXX(\lTP{s})}\dd s,
\end{equation}
which implies that $\Iv{\iUp + t}/n \to \lIv{t}$  in distribution along the original sequence.  The
convergence can be assumed almost sure by the Skorokhod coupling lemma, and so, since $\lIv{t}$ is continuous,
$\Iv{\iUp + t}/n \to \lIv{t}$ uniformly on $[-t_1,t_1]$,
\as{} for any $t_1 > 0$.

Take any $\eps > 0$. Then, \whp{} conditional on $\iUp < \infty$,
$\Iv{\iUp + t}/n < \aS - \fSv(\lTP{-t_1}) + \eps$
for any $t \le -t_1$, by \eqref{e:IvttoaI} and monotonicity of $\Sv{t}$.
Thus, \whp{} conditional on $\iUp < \infty$,
$\sup_{t\le -t_1}|\Iv{\iUp + t}/n -\lIv{t}|<2\eps$ if $t_1$ is
large enough.
Furthermore, the
argument 
leading to \eqref{e:bdItvt10} and \eqref{e:bdItvt11} gives
$\sup_{t\ge t_1}|\Iv{\iUp + t}/n-\lIv{t}| < 2 \eps$ \whp{} conditional on $\iUp < \infty$,
for $t_1$ large enough.

Hence $\Iv{\iUp + t}/n \pto \lIv{t}$ uniformly on $\R$ and
the remaining parts of
\eqref{e:convscSIR} follow.

\medskip
\noindent (e)
This is an immediate consequence of (d).

\medskip
\noindent (f)
Given $s \in (\fSv (\fXIroot ), \fSv (1))$, let
$T(s) \= \inf\{ t \ge 0: \Sv{t} \le n s \}$,
so that $\iUp = T(s_0)$.
We will show that, for any $s_1, s_2 \in (\fSv (\fXIroot ), \fSv (1))$
with $s_1< s_2$,
$\Prob (T(s_2) < \infty, T(s_1) = \infty) = o(1)$.
(Clearly, $\Prob (T(s_2) = \infty, T(s_1) < \infty) = 0$.) By part (e),
conditional on the event $T(s_2) < \infty$, $S_{\infty}/n \pto v_S (\fXIroot
)$, and so $S_{\infty}/n < s_1$ \whp{} conditional on $T(s_2) < \infty$,
which implies that, \whp{} conditional on $T(s_2) < \infty$,
$S_t \le n s_1$ for some $t <\infty$ and thus $T(s_1)<\infty$.
So $\Prob (T(s_2) < \infty, T(s_1) = \infty) = o(1)$,
as required.

Now, recall that $s_0 \in (\fSv (\fXIroot ), \fSv (1))$.
Let $\eps > 0$; by the above, $\Prob (\iUp = \infty, T(\fSv (1)- \epsilon/2)
< \infty) =o(1)$,  and so $\Prob (\iUp = \infty, S_t \le  n (\fSv (1)-
\epsilon/2) \mbox{ for some } t) =o(1)$. It then follows that
$\Prob (\iUp = \infty, S_{\infty} < n (\fSv (1)- \epsilon/2)) = o(1)$,
and so $\Prob (\iUp = \infty, S_0 - S_{\infty} > \eps n) = o(1)$.
Since $\eps$ can be taken arbitrarily small, we have that $S_0 - S_{\infty} = \op (n)$ on the event $\iUp = \infty$.



Since $\Sv0-\Sv\infty=\op (n)$ on $\iUp = \infty$, also
$\sup_{t\ge0}\Iv{t} \le \nI + \Sv0-\Sv\infty = \op (n)$ on $\iUp = \infty$.
Further,
the uniform summability of $\sum_{k=0}^\infty k \nSk/n$ (see \refR{r:UI})
implies that the number $\XS0-\XS\infty$ of half-edges that ever get infected is $\op(n)$ on $\iUp = \infty$ also.
Hence $\sup_{t \ge 0} \XI{t}=\op(n)$ on $\iUp = \infty$, and only $\op(n)$ half-edges pair off before the end of the epidemic, so
$\sup_{t \ge 0} (\X0-\X t)=\op (n)$ on $\iUp = \infty$.

\medskip
\subsubsection{Proof of part (c): there is a large epidemic with probability bounded away from zero}\label{s:large-epidemic}

To study the beginning phase of the epidemic, we concentrate on the number
$\xnit$ of free infective half-edges.  It is convenient to colour
the half-edges of a newly infected vertex according to whether the vertex recovers before
the half-edge can pair off.  More specifically, as soon as a new vertex is infected, we give it a random recovery
time with distribution $\Exp(\rho)$
and each of its remaining half-edges an infection time with distribution
$\Exp(\gb)$; these times are independent of each other and everything else.
We then colour each of these half-edges red if its infection time is smaller
than the recovery time of the vertex, and black otherwise. The black half-edges
thus will recover without pairing off, so we can ignore them, while the
red half-edges will pair off at some time.
Let $\znt$ be the number of free red half-edges at time $t$ (for this proof,
we need only consider
the process in its original time scale).

We fix a small $\eps>0$ and stop if and when either
$\xnit$ becomes at least $\eps n$ or at least
$\eps n$ infective half-edges have paired off.
Of course, we also stop if $\xnit$ becomes 0.
Before stopping, the number $\xnt$ of free half-edges thus satisfies
\begin{equation}\label{xnt}
  \sum_{k=0}^\infty kn_k = \xno \ge \xnt \ge \xno - 2\eps n.
\end{equation}
Furthermore,
before stopping, there are at least $\nsk-\eps n$ remaining susceptible
vertices of degree $k$, for each $k$. Hence, when an infective half-edge
pairs off, it will infect a vertex of degree $k$ with probability at
least
\begin{equation}
 \label{e:susc-degree-k-prob-bound}
\begin{split}
  \frac{k(\nsk-\eps n)}{\sum_{k=0}^\infty kn_k}
&=  \frac{k (\gas p_k +o(1)-\eps) n}{(\mu+o(1))n}
=\frac{k(\gas p_k-\eps)}{\mu}+o(1)
\ge \frac{k(\gas p_k-2\eps)}{\mu},
\end{split}
\end{equation}
if $n$ is large enough. (Uniformly in $k$ because we only have to consider
the finite set of $k$ for which the final expresssion is nonnegative.)
If a vertex of degree $k$ is infected, this produces $k-1$ new free infective
half-edges. Of these, some random number $Y_{k-1}$ will be red; the distribution
of $Y_{k-1}$ can easily be given explicitly, see \refL{L:Y} below,
but here we only need the simple fact that its mean
is
\begin{equation}\label{e:ey}
  \E [ Y_{k-1} ] = \frac{\gb}{\gb+\rho} (k-1).
\end{equation}
Since the infecting half-edge is removed from among the free
infective half-edges, the net change in $\znt$ is $Y_{k-1}-1$.

When a half-edge is paired off, there is also
a possibility that it
gets joined to another infective half-edge,
in which case that half-edge is
removed from the free infective half-edges and $\znt$ is reduced by 1 or 2
(depending on whether the second half-edge was black or red);
this happens with probability at
most
\begin{equation}
  \frac{\xnit}{\sum_{k=0}^\infty kn_k - 2\eps n}
\le   \frac{\eps n}{(\mu+o(1) - 2\eps) n}
< 2\eps/\mX
\end{equation}
if $n$ is large and $\eps$ small.

Finally, it may happen that the infective half-edge connects to a recovered
half-edge. In this case $\znt$ is reduced by 1.

Consequently, if $t_i$ is the $i$:th time an infective half-edge is paired off, then
\begin{equation}\label{Z>}
  \gD Z_{t_i} \ge -1 + U_i,
\end{equation}
where $U_i$ is independent of the previous history and
and has the mixture distribution
\begin{equation}
  U_i =
  \begin{cases}
	Y_{k-1} & \text{with probability  $k(\gas p_k-2\eps)_+/\mu$,
for each $k\ge1$},
\\
-1  & \text{with probability } 2\eps/\mX
\\
0 &\text{otherwise}.
  \end{cases}\label{e:Uidist}
\end{equation}
We may assume that $\eps\le \mX/2$, and then $U_i$ is well-defined, because
either $(\gas p_k-2\eps)_+=0$ for every $k\ge1$, or there exists a $k$ with
$\gas p_k>2\eps$, and then
\begin{equation*}
\sum_{k=0}^\infty k(\gas p_k-2\eps)_+
\le \sum_{k=0}^\infty k\gas p_k-2\eps
=\gas\lambda  -2\eps
\le\mu-2\eps.
\end{equation*}
The random variable $U_i$ has expectation
\begin{equation}\label{e:Ui}
  \begin{split}
  \E U_i &=\sumki \frac{k(\gas p_k-2\eps)_+}{\mu} \E Y_{k-1} -\frac{2\eps}{\mu}	
=\frac{\gb}{(\gb+\rho)\mu}\sumki (k-1){k(\gas p_k-2\eps)_+} -\frac{2\eps}{\mu}	
  \end{split}
\end{equation}
which, as $\eps\downto0$,  converges by monotone convergence
to
\begin{equation}
\frac{\gb}{(\gb+\rho)\mu}\sumki (k-1){k\gas p_k}
= \Rzero > 1.
\end{equation}
Thus we may assume that $\eps$ is chosen so small that $\E U_i>1$, and thus
$\E [U_i-1]>0$.

By \eqref{Z>}, the sequence $Z_{t_i}-Z_0$ ($i\ge1$), until we stop,
dominates  a random walk starting at $0$, with \iid{} increments
$U_i-1$ such that $\E [U_i -1]>0$. It is well-known that such a random walk is transient and
diverges to $+\infty$, so its minimum $M_-$ is a.s.\ finite.
If the process stops at time $t_i$ by running out of infective half-edges, we must have
$Z_{t_i}=0$, and thus $M_- \le -Z\nn_0$ (the $n$ dependency is added to $Z$ here as
a reminder that $M$ does \emph{not} depend on $n$).

Suppose that the initial number of infective half-edges
$\xnio = \sum_{k = 0}^\infty k \nIk \to\infty$ (however slowly).  Then also the initial number of red infective
half-edges $Z\nn_0\to\infty$ in probability; to see this, observe that either there are at least $\sqrt{\xnio}$
infective vertices with at least one half-edge (and the half-edges for different vertices are coloured independently),
or there is at least one infective vertex with at least $\sqrt{\xnio}$ half-edges (each of which was coloured independently,
given the recovery time).  Hence $\P(M_-\le -Z\nn_0)\to0$, and \whp{} we do not
stop by running out of red half-edges.

Since the process must stop at some point, \whp{} either the process stops by
producing at least $\eps n$ infective half-edges or because at least $\eps
n$ half-edges have paired off.  
In either case, at least one of $\xnit/n$ and $\xnt/n$ differs by at least
$\eps$  from its initial value,
so we cannot have convergence to the
trivial constant solution.
Thus \whp{} $\iUp<\infty$, see (f).
We deduce that $\XI{0} \to \infty$ implies
$\iUp < \infty$ \whpx.

Finally, provided there is initially at least one infective half-edge, we
have $\Prob(Z\nn_0 \ge 1)\ge \beta/(\rho + \beta)$.  Further, $Z\nn_0 \ge 1$ occurs
independently of the steps in the random walk above and it is straightforward to show $\P(M_- = 0) > 0$.
Hence, the probability of stopping by running out of red half-edges is at most
\begin{equation}
 \label{e:RWLBp}
 \Prob(M_- \le -Z\nn_0) \le 1 - \frac{\beta}{\rho + \beta}\Prob(M_- = 0) < 1.
\end{equation}
It follows that, with probability bounded away from zero, we stop by
producing at least $\eps n$ infective half-edges or because at least $\eps
n$ half-edges have been paired off.  As noted above, this implies that
$\xnit/n$ and $\xnt/n$ cannot converge to constants,
and thus \whp{} $\iUp<\infty$.

\section{The probability of a large outbreak and the size of a small outbreak}\label{s:branch}

In this section, we obtain an expression for the probability
$\Prob(\iUp < \infty)$
in \refT{t:mI0}\ref{t:mI0.ii.c} using a branching process approximation of the
initial steps.
We thus assume  \ref{d:first}--\ref{d:last},  $\aI=\mI = 0$ and $\Rzero>1$.
We begin by finding the distribution of the variables $Y_{k-1}$
defined in \refSSS{s:large-epidemic}.
We denote falling factorials by $(n)_j\=n\dotsm(n-j+1)$,
and let $\FF$ denote the hypergeometric function, as defined in \eg{}
\cite{MR2723248}.

\begin{lemma}\label{L:Y}
Consider an infective vertex with $\ell\ge0$ free half-edges and let
$Y_\ell$ be the number of them that pair off
 before the vertex recovers
(assuming that none of them is chosen by another infective half-edge in a
 pairing event).
If $\rho>0$, then
\begin{equation}\label{e:yell}
  \P(Y_\ell=j)=
\frac{\rho\,\ell!\,\gG(\ell+\rho/\gb-j)}{\gb\,(\ell-j)!\,\gG(\ell+\rho/\gb+1)}
=\frac{\rho}{\ell\gb+\rho}\cdot\frac{(\ell)_j}{(\ell+\rho/\gb-1)_j},
\end{equation}
and the \pgf{} $\ypgf{\ell}(x)\=\E x^{Y_\ell}$ is given by
\begin{equation}\label{e:ypgf}
\ypgf\ell(x)
=\sum_{j=0}^\ell \P(Y_\ell=j)x^j
=
\frac{\rho}{\ell\gb+\rho}\cdot\FF\Bigpar{-\ell,1;-\ell-\frac\rho\gb+1;x}.
\end{equation}

If $\rho=0$, then $Y_\ell=\ell$ and $\ypgf\ell(x)=x^\ell$.
\end{lemma}

\begin{proof}
The case $\rho=0$ is trivial, so assume $\rho>0$. By conditioning on the
time of recovery, and changing variables $x=e^{-\gb t}$
 to obtain a beta integral,
\begin{equation*}
  \begin{split}
\P(Y_\ell=j)
&= \intoo\binom{\ell}j\bigpar{1-e^{-\gb t}}^je^{-(\ell-j)\gb t}
e^{-\rho t}\rho \dd t	
\\
&= \intoi\binom{\ell}j\xpar{1-x}^j x^{\ell-j+\rho/\gb-1}\frac\rho\gb\dd x
\\
&= \frac\rho\gb\binom{\ell}j
 \frac{\gG(\ell-j+\rho/\gb)\gG(j+1)}{\gG(\ell+\rho/\gb+1)},
  \end{split}
\end{equation*}
which can be written as in \eqref{e:yell},
and \eqref{e:ypgf} follows from
the definition of hypergeometric functions.
\end{proof}

\begin{remark}
  The standard
hypergeometric distribution obtained by drawing $n$ balls from an urn
  with $N$ balls of which $m$ are white has \pgf{}
$c\cdot\FF(-n,-m;\allowbreak N-n-m+1;x)$,
where $c=(N-n)!\,(N-m)!/\bigpar{(N-m-n)!\,N!}$ is a normalization constant.
Hence the distribution of $Y_\ell$ can formally be
  regarded as a `negative hypergeometric distribution', with parameters
$(n,m,N)=(\ell,-1,-1-\rho/\gb)$.
\end{remark}

Next, we define $\xi$ to be the random variable obtained by mixing $Y_{k-1}$
with probabilities $\gas kp_k/\mu$; to be precise,
\begin{equation}\label{e:xi}
  \xi =
  \begin{cases}
	Y_{k-1} & \text{with probability  $\gas kp_k/\mu$, for each $k\ge1$},
\\
0 &\text{with probability $\mR/\mu$},
  \end{cases}
\end{equation}
where the probabilities sum to 1, since we assume $\mI=0$.
(If $\mR=0$, we thus mix using the size-biased distribution corresponding to
$(p_k)_0^\infty$.)
Note that this is the limiting case $\eps=0$ of \eqref{e:Uidist}.
We have, by \eqref{e:ey} and \eqref{e:R0}, \cf{} \eqref{e:Ui},
\begin{equation}\label{e:exi}
  \E\xi = \sum_{k=1}^\infty \frac{\gas kp_k}{\mu} \E Y_{k-1}
= \frac{\gas}{\mu}\cdot\frac{\gb}{\gb+\rho}\sum_{k=1}^\infty (k-1) kp_k
= \Rzero > 1.
\end{equation}

\begin{theorem} \label{T:GW1}
Suppose that the assumptions of \refT{t:mI0}(ii) (\refT{t:mI0-1}(ii)) hold.
Let $q\in(0,1)$ be the extinction probability for a Galton--Watson process
  with offspring distribution $\xi$, starting with a single individual.
Then for the SIR epidemic on the multigraph $\GrCM$ (simple graph $\Gr$)
\begin{equation}\label{e:gw}
\P(\iUp=\infty)
=
\prod_{k=1}^\infty \ypgf{k}(q)^{\nIk}+o(1).
\end{equation}
Hence the probability of a large outbreak is
$1-\prod_{k=1}^\infty \ypgf{k}(q)^{\nIk}+o(1)$.
\end{theorem}

Note that $\ypgf{k}(q)$ is the extinction probability if we start the
Galton--Watson process with $Y_k$ individuals instead of one. Hence the
product in \eqref{e:gw} is the extinction probability if we start with a
number of individuals distributed as $Y_k$
for each initially infected vertex of degree $k$, with these numbers
independent. (This is just the number of initially red infective half-edges in
\refS{s:large-epidemic}.)

Before proving \refT{T:GW1}, we state another theorem, saying that
a small outbreak infects only $\Op(1)$ vertices. This is valid in the
supercritical case without further assumption, and also in the (sub)critical
case
(when the outbreak is small \whp)
provided the initial number of infected half-edges
$\xnio = \sum_{k =  0}^\infty k \nIk$, is bounded.
(This is equivalent to the initial number of infected individuals and their
degrees  being
bounded.)

\begin{theorem} \label{T:GW2}
Suppose that the assumptions of \refT{t:mI0} (\refT{t:mI0-1}) are satisfied.

Then the following holds for the SIR epidemic on the multigraph $\GrCM$ (simple graph $\Gr$).

\begin{romenumerate}[-10pt]  
\item \label{tgw2.i}
If $\Rzero \le 1$
and further  the initial number of infective half-edges
$\xnio =O(1)$,
then the number $\nS - \Sv{\infty}$ of susceptible vertices that ever get
infected is $\Op(1)$.

\item \label{tgw2.ii}
If $\Rzero > 1$ , then
a small outbreak infects only $\Op(1)$ vertices.
In other words, for every $\eps>0$ there exists $K_\eps<\infty$ such that
\begin{equation}
  \label{e:gwO}
\P(\iUp=\infty \text{ and }\nS-\Sv\infty>K_\eps)<\eps.
\end{equation}
\end{romenumerate}
\end{theorem}

\begin{remark}
The  assumption $\xnio=O(1)$ is easily seen to be necessary in part (i) of this result.
Indeed, if $\xnio\to\infty$, then the number of initially infective half-edges
that pair off before recovering tends to infinity in probability.
While there are at least $\delta n$ susceptible vertices, then the probability that at a pairing event an infective half-edge is joined to a half-edge at a susceptible vertex is asymptotically bounded away from 0.
Hence, the number of susceptibles infected
tends to infinity in probability.


In the supercritical case this can be assumed because of
\refT{t:mI0}\ref{t:mI0.ii.c} (\refT{t:mI0-1}\ref{t:mI0.ii.c}).

We leave it to the reader to study the precise size of the epidemic in the
subcritical case when $\xnio\to\infty$.
The critical case ($\Rzero =1$)
is examined in a forthcoming work \cite{JLWcritical}.
\end{remark}

Theorems \ref{T:GW1} and \ref{T:GW2}
are valid both in the simple graph case and the multigraph case
of \refT{t:mI0}. We prove the multigraph case here and defer the simple
graph case to \refS{s:simple}.

\begin{proof}[Proof of Theorems \ref{T:GW1} and \ref{T:GW2} in the
	multigraph case]
We begin with the supercritical case $\Rzero>1$.
Since $\E\xi=\Rzero>1$ by \eqref{e:exi}, the Galton--Watson process is
  supercritical, and thus $0<q<1$ as asserted.

Fix a small $\eps>0$. We have shown in \refSSS{s:large-epidemic} that, if
$\XI0\to\infty$, then $\P(\iUp<\infty)\to1$.
It follows that there exists a finite $N$ (depending on the parameters
$\gb$, $\rho$, $(p_k)_0^\infty$, etc., but not on $n$) such that, if $\XI0\ge N$, then
$\P(\iUp<\infty)>1-\eps$. (If not, then there would exist a sequence of initial
conditions satisfying \ref{d:first}--\ref{d:last} with $\XI0\to\infty$,
and thus $\ntoo$, but $\P(\iUp)\le1-\eps$.)

Now consider the \GWp{} $(\hZ_i)_{i=0}^\infty$
with offspring distribution $\xi$ and some initial $\hZ_0$. Since the process is
supercritical, \as{} either $\hZ_i=0$ for some $i$, and thus also for all
larger $i$, or $\hZ_i\to\infty$ as $i\to\infty$.
We consider also the corresponding random walk stopped at 0, defined by
$\rw_0=\hZ_0$, $\rw_{i+1}=0$ if $\rw_i=0$ and $\rw_{i+1}=\rw_i-1+\xi_{i+1}$
if $\rw_i>0$, where $(\xi_i)$ are \iid{} copies of $\xi$. This can be
regarded as the \GWp{} modified so that different individuals give birth
at different
times, and thus $\hZ_i\to\infty\ (0)\iff \rw_i\to\infty\ (0)$.

A.s., $\rw_i\to\infty$ unless the random walk hits 0 and thus is absorbed there.
Thus, $\P(0<\rw_i<N)\to0$ as $i\to\infty$, and it follows
that there exists some $K<\infty$ such that,
for any initial number $\rw_0$,
$\P(0<\rw_K<N)<\eps$
and furthermore
$\P(\lim_{i\to\infty}\rw_i=\infty)>\P(\rw_K\ge N)-\eps$. (It suffices to
consider a finite number of $\rw_0$, since, if $\rw_0$ is large, then
$\P(\inf_i\rw_i < N)<\eps$.)
From now on, we let $\rw_0=\hZ_0$ be the initial number of red infective
half-edges.

Let us now return to the epidemic, and consider only the $K$ first half-edges that pair
off.
Let us colour the half-edges as in \refS{s:large-epidemic}. Since we consider only
a fixed number of steps,
each half-edge that pairs off infects a susceptible
vertex of degree $k$ with probability, \cf{} \eqref{e:susc-degree-k-prob-bound},
\begin{equation}
\begin{split}
  \frac{k(\nsk+O(1))}{\sum_{k=0}^\infty kn_k+o(n)}
= \frac{\gas k p_k}{\mu}+o(1).
\end{split}
\end{equation}
Hence we can \whp{} couple the first $K$ steps of the epidemic and the
random walk $\rw_i$ such that $Z_{t_i}=\rw_i$, for $i\le K$, where, as above
$Z_{t_i}$ is the number of red infective half-edges when  $i$ half-edges
have paired off.

If we have at least $N$ red infective half-edges after $K$ steps, let us
restart the epidemic there, and regard the situation after $K$ steps as a new
initial configuration, omitting the $2K$ half-edges that have been paired.
Since only $o(n)$ half-edges have changed status, the assumptions
\ref{d:first}--\ref{d:last} still hold; however, we now start with at least
$N$ infective half-edges. Thus, by our choice of $N$,
$\P(\iUp<\infty)>1-\eps$.

Combining these facts, we see that, with probability $1-O(\eps)+o(1)$, either:
\begin{romenumerate}
\item
the \GWp{} $(\hZ_i)$ dies out,
$\rw_K=0$, the epidemic infects at most $K$ vertices, and $\iUp=\infty$; or
\item
the \GWp{} $(\hZ_i)$ survives,
$\rw_K\ge N$, there are at least $N$ free infective half-edges after $K$
steps, and $\iUp<\infty$.
\end{romenumerate}
Since $\eps$ is arbitrary, \eqref{e:gw} follows, using the comment above
that  the extinction probability of $(\hZ_i)$ is the product in
\eqref{e:gw},
and so does \eqref{e:gwO}.

This proves the theorems in the supercritical case. In
 the (sub)critical case $\Rzero\le1$, we only have to prove
 \refT{T:GW2}\ref{tgw2.i}. This follows by comparison with a random walk as
 in the supercritical case; the details are simpler and are omitted.
\end{proof}

\begin{remark}
By standard branching process theory,
the extinction probability $q$ is the unique root in $(0,1)$ of
  \begin{equation}
	q=\E q^\xi
=
\frac{\mR}{\mu}+\frac{\gas}{\mu}\sum_{k=1}^\infty kp_k\ypgf{k-1}(q),
  \end{equation}
where $\ypgf{k-1}$ is given by \eqref{e:ypgf}. Hence $q$
and the (asymptotic) probability
$1-\prod_{k=1}^\infty \ypgf{k}(q)^{\nIk}$
of a large outbreak
can be determined to arbitrary precision given the parameters of the model.
\end{remark}

\begin{remark}
  We have in \refT{T:GW1} used a single-type Galton--Watson process, for
  simplicity.
It is perhaps more natural to use a multi-type Galton--Watson process, with
types $0,1,2,\dots$, where an individual of type $k$ has a number of
children distributed as $Y_{k}$ and each child is randomly assigned
type $\ell$ with probability
$\gas (\ell+1) p_{\ell+1}/\mu +\gd_{\ell,0}\mR/\mu$, $\ell\ge0$;
we start with $\nIk$ individuals of type $k$ for each $k \ge 0$. The total number of
individuals in each generation, ignoring their types, will form the
single-type Galton--Watson process above, with offspring distribution $\xi$
(starting as explained above).
Hence the two different branching processes have the same extinction
probability, which is the product  in \eqref{e:gw}.
\end{remark}

\section{Transfer to the simple random graph: proofs of Theorems~\ref{t:mIg0-1} and~\ref{t:mI0-1}}\label{s:simple}

In this section we consider the simple random graph $\Gr$.
We assume \ref{d:sumsquaredi=On},
in addition to \ref{d:first}--\ref{d:last}.

All results in Theorems \ref{t:mIg0} and \ref{t:mI0}
about convergence in probability, or results holding \whp, immediately transfer
from the random multigraph $\GrCM$ to the simple graph $\Gr$ by conditioning
on $\GrCM$ being simple, since $\liminf_{\ntoo}\P(\GrCM\text{ is simple})>0$ by
assumption \ref{d:sumsquaredi=On}
and  \cite{Janson:2009:PRM:1520305.1520316}.

It remains to show \refT{t:mI0-1}\ref{t:mI0.ii.c}, and the more precise
\refT{T:GW1}, for $\Gr$.
Again, the case $\xnio\to\infty$, when $\P(\iUp<\infty)\to1$, is clear.
By considering subsequences, it thus suffices to consider the case $\xnio=O(1)$, \ie{} a bounded number of initially infected half-edges.
Our assumptions allow a number $\nIx0$ of isolated infected vertices, but
they do not affect the edges in the graph or the infections at all, so we
may simply delete them and assume $\nIx0=0$.
(Since we assume $\nI/n\to\aI=0$,  the number of isolated infected vertices
$o(n)$ and their deletion will not affect the assumtions.)
We may thus assume that there is a bounded number of initially infected
vertices, with uniformly bounded degrees. By again considering  subsequences,
we may assume that the numbers $\nIk$ of initially infected vertices are
constant (\ie, do not depend on $n$), with only a finite number different
from 0.
This assumption implies that the number
$\QQ\=\prod_{k=1}^\infty \ypgf{k}(q)^{\nIk}$ in \refT{T:GW1} is constant,
and the conclusion \eqref{e:gw} may be written
\begin{equation}
\P(\iUp=\infty)\to \QQ.
\end{equation}
We have shown this for the random multigraph $\GrCM$, and
we want to show that it holds also if we condition on $\GrCM$ being simple,
\ie, that the events $\set{\iUp=\infty}$ and \set{\GrCM \text{ is simple}}
are asymptotically independent.

We thus continue to work with $\GrCM$ and the configuration model.
Let $W$ be the number of loops and pairs of parallel edges in $\GrCM$;
thus $\GrCM$ is simple if and only if $W=0$. We write $W=\sum_{\ga\in\cA}I_\ga$,
where the index set $\cA=\cAloop\cup\cApair$ with $\cAloop$ the set of all pairs
of two half-edges at the same vertex and $\cApair$ the set of all pairs of two
pairs \set{a_i,a_j} and \set{b_i,b_j} with $a_i$ and $b_i$ distinct
half-edges at some vertex $i$ and $a_j$ and $b_j$ distinct half-edges at
some other vertex $j$; $I_\ga$ is the indicator that these
half-edges form a loop or a pair of parallel edges, respectively.

We let $\cL$ be the event that at most $\log n$ vertices will be infected.
By the definition of $\iUp$, $\cL$ implies $\iUp=\infty$ (for large $n$),
and by \refT{T:GW2}\ref{tgw2.ii}, $\P(\iUp=\infty \text{ and not }\cL)=o(1)$.
Hence $\cL$ and \set{\iUp=\infty} coincide \whp, and it suffices to show
that
\begin{equation}\label{cl|w}
\P(\cL\mid W=0)\to \QQ.
\end{equation}
We do this by inverting the conditioning and using the method of moments, in
the same way as in the proof of a similar result in
\cite{janson2008asymptotic}.
(See in particular the general formulation in
\cite[Proposition 7.1]{janson2008asymptotic}. However, we prefer to give a
self-contained proof here.)

First,
since we already know $\P(\cL)\to \QQ>0$ and
$\liminf\P(W=0)>0$,  \eqref{cl|w} is
equivalent to
\begin{equation}
  \P(\cL \text{ and }W=0)=\P(\cL)\P(W=0)+o(1)
\end{equation}
and thus to
\begin{equation}\label{e:PW0}
\P(W=0\mid\cL)=\P(W=0)+o(1).
\end{equation}

By again considering a subsequence, we may assume that
$W\dto \WW$ for some random variable $\WW$, with convergence of all moments,
where furthermore the distribution of
$\WW$ is determined by its moments (at least among
non-negative distributions), see
\cite{Janson:2009:PRM:1520305.1520316} if the maximum degree
$\max d_i=o(n\qq)$ and \cite{simpleII} in general.
($\WW$ has a Poisson distribution if
$\max d_i=o(n\qq)$, but not in general, see \cite{simpleII}.)

We write
$\cA=\cAa\cup\cAb$,
where $\cAb$ is the set of
all $\ga\in\cA$ that include a half-edge at an initially infected vertex,
and $\cAa$ are the others. Correspondingly, we have
$W=\Wa+\Wb$, where $\Wa\=\sum_{\cAa}I_\ga$
and $\Wb\=\sum_{\cAb}I_\ga$.

Since the number of initially infected
half-edges is $O(1)$,
we have, using \eqref{e:sumsquaredi=On} and
denoting the total number of half-edges by $N\=\sum_k k\nk\sim \mu n$,
\begin{equation}
  \E \Wb \le  \frac{O(1)}{N-1} + \frac{O(1)\sum_i d_i^2}{(N-1)(N-3)},
\end{equation}
where the first term on the \rhs{} is an upper bound on the number of loops at infective vertices, and the second term is an upper bound on the number of pairs of parallel edges coming out of infective vertices. (The probability that a particular half-edge at an infective vertex joins to vertex $i$ is $d_i/(N-1)$, and the probability that another half-edge at the same infective vertex joins to vertex $i$ is $(d_i-1)/(N-3)$.)
Thus $\Wb\pto0$, and $\Wa=W-\Wb\dto\WW$. For each $m\ge1$, $\E\Wa^m\le\E
W^m\to\E\WW^m$, so each moment of $\Wa$ is bounded as \ntoo, and thus
the convergence $\Wa\dto\WW$ implies moment convergence
\begin{equation}
  \label{eWam}
\E\Wa^m\to\E\WW^m
\end{equation}
for every $m\ge1$.

We next consider the conditioned variable
$(\Wa\mid\cL)$. We want to show that $(\Wa\mid\cL)\dto\WW$ by the
method of moments, so we fix $m\ge1$ and write
\begin{equation}
\E(\Wa^m\mid\cL)
=\sum_{\ga_1,\dots,\ga_m\in\cAa}\E( I_{\ga_1}\dotsm I_{\ga_m}\mid\cL).
\end{equation}
Consider some $\ga_1,\dots,\ga_m\in\cAa$. If we condition on
$I_{\ga_1}\dotsm I_{\ga_m}=1$, we have fixed the pairing of $O(1)$
half-edges, but the remaining half-edges are paired uniformly, so if we remove
the edges given by ${\ga_1},\dots, {\ga_m}$, we have another instance of the
configuration model, say $\Grxx$.
Obviously, our assumptions \ref{d:first}--\ref{d:last}
and \ref{d:sumsquaredi=On} hold for $\Grxx$ too. Furthermore, the
probability that the infection will spread to any vertex involved in
${\ga_1},\dots, {\ga_m}$ before $\log n$ vertices have been infected is
$O\bigpar{\log n\max_i d_i/n}=o(1)$, since $\max_id_i=O(n\qq)$
by \ref{d:sumsquaredi=On}. Hence, \whp{}, the extra edges given by
${\ga_1},\dots, {\ga_m}$ will not affect whether $\cL$ occurs or not.
Consequently, \refT{T:GW1} applied to $\Grxx$ shows that
\begin{equation}\label{e:l|ia1m}
  \P(\cL\mid I_{\ga_1}\dotsm I_{\ga_m}=1)
=\P(\cL \text{ in }\Grxx)+o(1)
=\QQ+o(1),
\end{equation}
uniformly in all ${\ga_1},\dots, {\ga_m}\in\cAa$ (for a fixed $m$).
(The estimate \eqref{e:l|ia1m} fails if some $\ga_i\in\cAb$,
for example if $\ga_1$ is a loop at an
initially infected vertex. This is the reason for considering $\Wb$
separately.)
We invert the conditioning again and obtain,
uniformly in all ${\ga_1},\dots, {\ga_m}\in\cAa$,
\begin{equation}
  \begin{split}
\E(I_{\ga_1}\dotsm I_{\ga_m}\mid\cL)
&=\frac{\P(I_{\ga_1}\dotsm I_{\ga_m}=1\text{ and }\cL)}{\P(\cL)}
\\&
=\frac{(\QQ+o(1))\P(I_{\ga_1}\dotsm I_{\ga_m}=1)}{\QQ+o(1)}
\\&
=\bigpar{1+o(1)}\E(I_{\ga_1}\dotsm I_{\ga_m}).
  \end{split}
\end{equation}
Summing over all ${\ga_1},\dots, {\ga_m}\in\cAa$ yields,
using \eqref{eWam},
\begin{equation}
\E(\Wa^m\mid\cL)
=\bigpar{1+o(1)}\E(\Wa^m)\to\E\WW^m.
\end{equation}
This holds for each $m\ge1$, and thus
$(\Wa\mid\cL)\dto\WW$ by the method of moments.
Furthermore, $\Wb\pto0$, and thus $(\Wb\mid\cL)\pto0$.
Consequently,
$(W\mid\cL)=(\Wa+\Wb\mid\cL)\dto\WW$. In particular,
$\P(W=0\mid\cL)\to\P(\WW=0)$.
Since also $W\dto\WW$ and thus $\P(W=0)\to\P(\WW=0)$,
the equation \eqref{e:PW0} follows, which, as stated above, implies \eqref{cl|w}
and completes the proof.

\section{No second moment assumption}\label{ss:nosecondmoment}
The main reason for making the assumption \ref{d:sumsquaredi=On} on a
bounded second moment of the degree distribution
is that it allows us to study the epidemic on the configuration model multigraph
instead of the uniform simple graph, see \refS{s:simple}.
However, we have used \ref{d:sumsquaredi=On} also in some parts of the proofs
for the multigraph case. We will now show that this is not necessary, so that,
in the multigraph case, the only assumption on the
degree distribution required is the uniform summability of $\sum_k k \nSk/n$
inherent in
assumptions \ref{d:alphaprops}--\ref{d:suscmeanUS}, see \refR{r:UI}.
The uniform summability assumption seems indispensible.

In the proof of \refT{t:mIg0}\ref{t:mIg0.b},
using an identical argument but
without assuming \ref{d:sumsquaredi=On},
$\fXI$ and $\pI$ are Lipschitz
continuous on $[\fXIroot,\theta_1]$ for any $\theta_1<1$,
so that  $\pI$ is locally Lipschitz
continuous on $[\fXIroot,1)$.
Hence, if we
choose $b\in(\fXIroot,1)$, then there is a unique solution $\tth_t$, $t\ge0$,
to the differential equation with initial condition $\tth_0=b$, and
we may extend $\tth_t$ to a solution  $\tth_t:(a,\infty)\to(\fXIroot,1)$
for some $a<0$,
where, by
choosing $|a|$ maximal, $\tth_t\to1$ as $t\downto a$.
Note that $a>-\infty$, since $\pI(1)>0$ and $\pI$ is continuous.
The translate
$\theta_t\=\tth_{t+a}$ then is a solution
to \eqref{e:dlTPtmIg0}. Uniqueness is proved similarly: if $\theta_t$ and
$\bar\theta_t$ are two solutions to \eqref{e:dlTPtmIg0}, then we may
translate them  so that the translates
$\theta_{t-a}$ and $\bar\theta_{t-\bar a}$ (with $a,\bar a<0$)
have the same value $b$
at 0, with $b\in(\fXIroot,1)$. By the uniqueness, as long as
$\theta\in[\fXIroot,1)$, the two translates
$\theta_{t-a}$ and $\bar\theta_{t-\bar a}$
coincide on the set where they
are both less than 1, and, by continuity, they coincide completely and thus
$a=\bar a$ and $\theta_t=\bar\theta_t$.
The rest of the proof is the same.

In the proof of \refT{t:mI0}\ref{t:mI0.ii.b},
we note first that, as
in the proof of \refT{t:mIg0}\ref{t:mIg0.b} just given,
$\pI$ is locally Lipschitz continuous on $[\fXIroot,1)$, which implies
uniqueness of the solution to \eqref{e:dlTPt}.
Furthermore, as before,
we define $\lttInv{}$ by \eqref{e:lttInv2} and note that
\begin{equation}\label{e:lttInv0}
\lim_{\tau\to0} \lttInv{\tau}
=
 \lttInv{0}
\= -\int_0^{\taz} \frac{ d\sigma} {\beta \pI(e^{-\sigma}) }.
\end{equation}
If $\lttInv0=-\infty$, everything is as before.
However,
if we do not assume \ref{d:sumsquaredi=On}, it is possible that
the integral converges and thus
$\lttInv0>-\infty$.
In this case, the inverse $\ltt:[\lttInv0,\infty) \to [0,\fpT)$ of
	$\lttInv{}$ is continuously differentiable on $(\lttInv0,\infty)$, and,
	if we extend it by defining $\ltt(t)=0$ for $t<\lttInv0$, then
it is continuous on $\R$ and satisfies \eqref{ltt'}
for all $t\neq\lttInv0$. (Recall that $\pI(1)=0$.)
Since the right-hand side of \eqref{ltt'} is continuous on $\R$,
it follows that $\ltt(t)$  is
continuously differentiable everywhere
and that \eqref{ltt'} is satisfied also
at $\lttInv0$.
Thus, again, $\lTP{t} = \exp(-\ltt(t))$
satisfies \eqref{e:dlTPt}, although now $\lTP{t}=1$ for $t\le\lttInv0$.
(In this case, there is a unique solution to \eqref{e:dlTPt} for
$\theta_0<1$ but infinitely many solutions  for $\theta_0=1$.)
The rest of the proof is the same.

The proof of \refT{t:mI0}\ref{t:mI0.ii.d} remains the same if
$\lttInv0=-\infty$.
If $\lttInv0>-\infty$, then the interval $[-t_1,t_1]$ should be replaced by
$[t_2,t_1]$ with $-\lttInv0<t_2<t_1<\infty$.
The convergence \eqref{e:xtau0} then follows as before, uniformly on each
such $[t_2,t_1]$. Furthermore, the same monotonicity argument as before, now
using $\ltt(t_2)\to0$ as $t_2\to\lttInv0$,
shows that \eqref{e:xtau0} holds uniformly on $\R$. The rest is the same as
before.

Let us confirm it is indeed possible that $\hat A_0>-\infty$.
By \eqref{e:lttInv0} and \eqref{e:pI},
\begin{equation}\label{xyz}
 \lttInv{0} >-\infty
\iff
\int_0^{\taz} \frac{ d\sigma} { \pI(e^{-\sigma}) } <\infty
\iff
\int_0^{\taz} \frac{ d\sigma} { \fXI(e^{-\sigma}) } <\infty
\iff
\int_0^{x_0} \frac{ dx} { \fXI(1-x) } <\infty,
\end{equation}
for any small positive $x_0$ ($x_0<1-\fXIroot$).
Recall that $\fXI(1)=\mI=0$. Thus, \eqref{e:fX}--\eqref{e:fXI} imply that,
as $x\to0$,
$\fXI(1-x)=\fXS(1)-\fXS(1-x)+O(x)$. It follows easily from \eqref{xyz} that
\begin{equation}\label{xyzw}
 \lttInv{0} >-\infty
\iff
\int_0^{x_0} \frac{ dx} {\fXS(1)- \fXS(1-x) } <\infty.
\end{equation}

\begin{example}\label{Efinite}
  Suppose that $p_k\sim k^{-\ga-1}$ as \ktoo, with $1<\ga\le 2$.
(This means that the asymptotic degree distribution has moments of order
$<\ga$, but not of order $\ga$.) Then \eqref{e:fXS} implies that if $1<\ga<2$,
then $\fXS'(1-x)\asymp x^{\ga-2}$ as $x\to0$ and thus
$\fXS(1)-\fXS(1-x)\asymp x^{\ga-1}$
so the integral in \eqref{xyzw} converges and $\lttInv0>-\infty$.
However, if $\ga=2$, then
$\fXS'(1-x)\asymp |\log x|$,
so $\fXS(1)-\fXS(1-x)\asymp x |\log x|$ and $\lttInv0=-\infty$.
\end{example}

We end this section by noting that,
when there is no second moment, then $\Rzero$ is infinite (and thus the
process is supercritical, for any $\beta$ and $\rho$). On the other hand, if
\ref{d:sumsquaredi=On} holds,
then $\Rzero$ is finite, \cf{} Remark~\ref{r:2ndmoment}.



\section{The random time shift}
\label{s:time-shift}

In Theorems \ref{t:mI0}(ii) and~\ref{t:mI0-1}(ii), we use a (random) time-shift $\iUp$,
which can be interpreted as the time it takes for the epidemic to grow from
a small number of initially infected to a large outbreak.

Suppose for simplicity that  $\sum_{k = 1}^\infty k\nIk \to \infty$,
so $\iUp<\infty$ \whp{} by \refT{t:mI0}(ii) for $\GrCM$, and by \refT{t:mI0-1}(ii) for $\Gr$.
Then, for any $\gd>0$, \eqref{e:addFuncConvTTShift} shows that
\begin{equation}
  \iUp \ge \iUp-A_\gd \pto -\lttInv{\gd}.
\end{equation}
Suppose first that \ref{d:sumsquaredi=On} holds.
Letting $\gd\to0$, we have $-\lttInv\gd\to\infty$,
as noted after \eqref{e:lttInv2}; thus $\iUp\pto\infty$.

On the other hand, if we assume that
\ref{d:sumsquaredi=On} does not hold (and we thus consider the multigraph
case), the situation is more complicated. It is possible that
$\lttInv0=-\infty$, as seen in \refE{Efinite},
and then $\iUp\pto\infty$ as above.
However, if $\lttInv0>-\infty$, then we obtain only $\iUp \ge |\lttInv0| -\eps$
\whp, for every $\eps>0$.
It is possible to give examples showing that in this case, both
$\iUp\pto\infty$
and $\iUp\pto|\lttInv0|<\infty$ are possible (as well as intermediate cases).
We omit the details.

\appendix

\section{Time-changed Markov chains}\label{Atimechange}

The following lemma justifies the time change in \refS{s:fastepidemic}. A more general
result appears in \cite[III. (21.7)]{RogersWilliams} but we include a proof
in the simpler setting of Markov chains.

\begin{lemma}\label{l:AddFuncTT}
Suppose $(Y_\tau)_{\tau \ge 0}$ is a continuous time Markov chain with finite state space $E$, and infinitesimal
transition rates $(q(i,j))_{i,j \in E}$.

Let $f:E \to (0,\infty)$ be strictly positive and define the strictly increasing process
\begin{equation}\label{e:lAddFunc}
 A_\tau = \int_0^\tau f(Y_\sigma)\dd\sigma, \qquad \tau \ge 0,
\end{equation}
and its inverse $\tau(t)$, $t \ge 0$.

Then the time-changed process $(Y_{\tau(t)})_{t \ge 0}$ is again Markov and has
infinitesimal rates
\linebreak
$(q(i,j)/f(i))_{i,j \in E}$.
\end{lemma}

\begin{proof}
The paths of $Y$ are piecewise constant and right continuous.
Let $J_0 = 0$ and
$J_{k+1} = \inf\{\tau \ge J_k : Y_{\tau} \neq Y_{J_k}\}$, $k \ge 0$, be the jump times
 of $(Y_{\tau})_{\tau \ge 0}$.  The state space $E$ is finite, so
 the process is non-explosive and \as{} $J_k \to \infty$.
Also, without loss of generality we may assume no state is absorbing, \ie{} the rate of leaving $q(i) \= - q(i,i) = \sum_{j \neq i} q(i,j) > 0$
 is strictly positive for each $i \in E$.  Thus $J_k < \infty$ \as{} for
 each $k$.


For notational ease, let $\tilde Y_t \= Y_{\tau(t)}$,  $t \ge 0$.  Since $\tau(A_{J_k}) = J_k$, the jump times of
$\tilde Y$ are given by $\tilde J_k \= A_{J_k}$, $k \ge 0$.
It follows that the holding times for $\tilde Y$ are
\begin{align*}
\tilde J_{k+1} - \tilde J_k & =  A_{J_{k+1}} - A_{J_k} \\
& = (J_{k+1} - J_k)f(Y_{J_k}) \\
  & = \left( (J_{k+1} - J_k)q(Y_{J_k}) \right) \parfrac{f(Y_{J_k})}{q(Y_{J_k})} \\
  & = T_k \; \tilde  q(\tilde Y_{\tilde J_k})^{-1} 
\end{align*}
$k \ge 0$, where $T_{k} = (J_{k+1} - J_k)q(Y_{J_k})$ and $\tilde q(i) = q(i)/f(i)$, $i \in E$.  The Markov property
 of $(Y_{\tau})_{\tau \ge 0}$ implies the $T_k$ are
independent $\Exp(1)$ random variables that are also independent of the embedded jump
chain $(Y_{J_k})_{k \ge 0} = (\tilde Y_{\tilde J_k})_{k \ge 0}$.   The latter has transition
kernel $(q(i,j)/q(i))_{i,j \in E} = (\tilde q(i,j)/\tilde q(i))_{i,j \in E}$,
where $\tilde q(i,j) = q(i,j)/f(i)$, $i,j \in E$.
It follows that $\tilde Y_t$ is a Markov chain with transition rates $\tilde
q(i,j)$.
\end{proof}

\section{Summary of notation}\label{Anotation}

For ease of reference we summarise the main notation used to describe the epidemic.
A subscript of $\mathrm{S}$, $\mathrm{I}$ or $\mathrm{R}$ always refers to susceptible, infective or recovered, respectively.
The initial conditions are as follows.
\smallskip

\begin{center}
\begin{tabular}{p{5cm}p{9cm}}
$\nS$, $\nI$, $\nR$ 				&	 Number of vertices (of given type). \\
$\nSk$, $\nIk$, $\nRk$ 				&	 Number of degree $k \ge 0$ vertices. \\
$\aS$, $\aI$, $\aR$					&	 Limiting fractions of vertices. \\
$(p_k)_{0}^\infty$, $\lambda$ 		&	 Limiting degree distribution for a randomly chosen susceptible and its mean. \\
$\mX = \mS + \mI + \mR$				&	 Limiting mean degree (for any vertex). \\
\end{tabular}
\end{center}

\smallskip

The stochastic processes are denoted as follows.

\begin{center}
\begin{tabular}{p{5cm}p{9cm}}
$\Sv{t}$, $\Iv{t}$, $\Rv{t}$ 		&	 Number of vertices (of given type) at time $t\ge 0$. \\
$\Sv{t}(k)$, $\Iv{t}(k)$, $\Rv{t}(k)$ 		&	Number of degree $k \ge 0$ vertices. \\
$\XS{t}$, $\XI{t}$, $\XR{t}$ 		&	Number of free half-edges. \\
$\X{t} = \XS{t}+\XI{t}+\XR{t}$ 		&	Total number of half-edges. \\
\end{tabular}
\end{center}

\smallskip
Time-changed versions of the above processes are superscripted with a prime and use greek time indices.
In addition we have $\tau(t)$ and $A_\tau$, the time change and its inverse.
The limiting trajectories for these processes are denoted as follows.
\begin{center}
\begin{tabular}{p{5cm}p{9cm}}
$\ltt(t)$, $\lttInv{\tau}$										&	Time change and its inverse.\\
$\lTP{t} =  e^{-\hat\tau(t)}$ 								&	Parameterisation of time. \\
$\fSv(\lTP{t})$, $\lIv{t}$, $\lRv{t}$						&	Fraction of vertices (of given type) at time $t\ge 0$. \\
$\fXS(\lTP{t})$, $\fXI(\lTP{t})$, $\fXR(\lTP{t})$,
$\fXX(\lTP{t})$  		&	Number of free half-edges divided by $n$. \\
\mbox{$\pS(\lTP{t}) = \fXS(\lTP{t})/\fXX(\lTP{t})$},
& Susceptible proportion of free half-edges.
\\
$\pI(\lTP{t})= \fXI(\lTP{t})/\fXX(\lTP{t})$ & Infective	proportion of free half-edges.
\end{tabular}
\end{center}
\smallskip

The following quantities are also important in our analysis.
\begin{center}
\begin{tabular}{p{5cm}p{9cm}}
$\fXIroot = \lim_{t\to\infty}\lTP{t}$ 	& Root of $\fXI$ corresponding to the end of the epidemic. \\
$\iEndtt$ 								& Duration of the time-changed epidemic.\\
$\hat\tau_{\infty} = -\ln \fXIroot$ 	& Limiting duration of the time-changed epidemic, assuming a large outbreak.\\
$\calibS \in (\fSv(\fXIroot),\fSv(1))$	& Fraction of susceptibles used in calibration time $T_0$. \\
\end{tabular}
\end{center}

\bibliography{sir-aug2014}
\bibliographystyle{plain}

\end{document}